\documentclass[11pt,leqno]{amsart}
\usepackage{amsmath}
\usepackage{amssymb}
\usepackage{amsthm}
\usepackage{graphicx,color,subfigure}
\usepackage{enumerate}
\usepackage{multirow}
\usepackage{multicol}
\usepackage{tabularx}

\newtheorem{theorem}{Theorem}[section]
\newtheorem{lemma}{Lemma}[section]
\newtheorem{corollary}{Corollary}[section]

\newtheorem{proposition}{Proposition}[section] 
\newtheorem{definition}{Definition}[section]
\newtheorem{remark}{Remark}[section]
\newtheorem{example}{Example}[section]

\renewcommand{\arraystretch}{1.2}
\newtheorem{alg}{Algorithm}

\def\ba{\begin{array}}
\def\ea{\end{array}}
\def\beq{\begin{equation}}
\def\eeq{\end{equation}}
\def\bea{\begin{eqnarray}}
\def\eea{\end{eqnarray}}
\def\beann{\begin{eqnarray*}}
\def\eeann{\end{eqnarray*}}

\def\R{\mathbb{R}}

\def\Z{\mathbb{Z}}

\def\dist{\textup{dist}}

\def\inte{\textup{int}}

\def\range{\textup{range}}

\def\argmin{\textup{argmin}}

\def\rec{\textup{rec}}

\def\exp{\textup{exp}}

\newif \ifcomment 

\commentfalse

\title{\bf Status Determination by Interior-Point Methods for Convex Optimization Problems in Domain-Driven Form}
\author{
Mehdi Karimi \and Levent Tun\c{c}el}
\thanks{* Some of the material in this manuscript appeared in a preliminary form in Karimi's PhD thesis \cite{karimi_thesis}. \\
Mehdi Karimi (m7karimi@uwaterloo.ca) and Levent Tun\c{c}el (ltuncel@math.uwaterloo.ca): Department of
Combinatorics and Optimization, University
of Waterloo, Waterloo, Ontario N2L 3G1, Canada. Research of the
authors was supported in part by Discovery Grants from NSERC and by U.S. Office of Naval Research under award numbers: N00014-12-1-0049, N00014-15-1-2171 and N00014-18-1-2078.}

\begin{document}
\begin{abstract}
We study the geometry of convex optimization problems given in a Domain-Driven form and categorize possible statuses of these problems using duality theory. Our duality theory for the Domain-Driven form, which accepts both conic and non-conic constraints, lets us determine and certify statuses of a problem as rigorously as the best approaches for conic formulations (which have been demonstrably very efficient in this context).  
 We analyze the performance of an infeasible-start primal-dual algorithm for the Domain-Driven form in returning the certificates for the defined statuses. 
Our iteration complexity bounds for this more practical Domain-Driven form match the best ones available for conic formulations. 
At the end, we propose some stopping criteria for practical algorithms based on insights gained from our analyses. 
\end{abstract}
\maketitle

\pagestyle{myheadings} \thispagestyle{plain}
\markboth{KARIMI and TUN{\c C}EL}
{Status Determination for Domain-Driven Formulations}



\section{Introduction} \label{introduction}
In this article, we are interested in \emph{convex optimization} as minimizing a \emph{convex function} over a closed \emph{convex set} in a finite dimensional Euclidean space. Without loss of generality, we may assume that the objective function is linear, then,  an instance $(P)$ of a convex optimization problems can be written as 
\begin{eqnarray} \label{main-p}
(P) \ \ \ \inf _{x} \{\langle c,x \rangle :  Ax \in D\},
\end{eqnarray}
where $x \mapsto Ax : \R^n \rightarrow \R^m$ is a linear embedding, $A$ and $c \in \mathbb \R^n$ are given, and $D \subset \mathbb \R^m$ is a closed convex set. For problem $(P)$, there are four possible statuses:
\begin{itemize}
\item having an optimal solution, 
\item having a finite optimal value, but no optimal solution, 
\item being unbounded (for every $M \in \R$, there exists a feasible solution with objective value strictly better than $M$),
\item being infeasible. 
\end{itemize}
The second status above cannot happen for the linear programming (LP) special case, which shows that determination of the statuses for a general convex optimization problem is more complicated than that for LP. 
A dual problem can be assigned to every $(P)$ and the above four statuses are also possible for this dual, hinged on the primal ones. The connection between the primal and dual problems  lets us provide verifiable \emph{certificates} for the statuses. For a discussion on the possible status patterns in the context of duality, see \cite{interior-book}-Section 4.2.2.\ or \cite{luo1996duality,MR1811754}. 
We can split the instances of convex optimization problems into ``well-posed", which means there exists an $\epsilon >0$ such that every $\epsilon$-perturbation\footnote{$\epsilon$-perturbation of $c$ means replacing $c$ by $c' \in \R^n$ where $\|c-c'\| \leq \epsilon$ (similarly for $A$), and $\epsilon$-perturbation of $D$ means shifting it by a vector $b\in \R^m$ with $\|b\| \leq \epsilon$.} of $(c,A,D)$ has the same status as that of $(c,A,D)$, or ``ill-posed", which means there exists an arbitrarily small perturbation of $(c,A,D)$ that changes the status of the instance. Renegar proved that (\cite{renegar1994possible}-Theorem 1.4.1) for ``semi-algebraic problems", it is impossible to know whether a given instance is ill-posed if one is using error measurement functions that are also semi-algebraic. Renegar  also developed complexity results based on the distance of the problem to being ill-posed \cite{renegar1994some, renegar1995incorporating}. While being aware of the above fundamental results, practical and theoretical approaches for solving convex optimization problems, given an input data set,  should return a status and a certificate, and must strive to determine these two as rigorously as possible.

After categorizing the statuses for a formulation, the next question is how does an algorithm determine the status of a given problem instance, in theory and in practice. Iterative algorithms initiate the process of solving a problem from a \emph{starting point}. 
In the development of most of the theory of such algorithms, a large portion of the literature has focused on feasible-start case (i.e. a feasible solution is assumed available to start the algorithm).  For most of the theoretical development, this is sufficient since one can employ a feasible-start algorithm in a two-phase approach or in other standard approaches. 
 For some applications, such as the recently popular implementation of interior-point methods in designing fast algorithms for combinatorial problems \cite{christiano2011electrical, cohen2017negative}, an obvious feasible solution is available. However, in general applications of and software for convex optimization, infeasible-start algorithms are essential. 
We review some of the infeasible-start interior-point methods for LP and general conic optimization in Section \ref{sec:review}. 

Definition of possible statuses for a given problem should depend on the formulation being used, and mostly on the way the underlying convex set is given. In this article, we consider the \emph{Domain-Driven} form for convex optimization \cite{karimi_arxiv,karimi_thesis, cone-free}, which is the form in \eqref{main-p} where  $D$ is
defined as the closure of the domain of a given $\vartheta$-\emph{self-concordant (s.c.) barrier} $\Phi$ \cite{interior-book}. Since every open convex set is the domain of a s.c.\ barrier \cite{interior-book}, in principle, every convex optimization problem can be treated in the Domain-Driven setup. 
The Domain Driven form is introduced to extend many desirable properties of primal-dual interior-point techniques available for conic optimization to a form that does not require all nonlinear constraints to be cone constraints \cite{karimi_arxiv}.
In applications, the most restrictive part of the modern interior-point approach is that a ``computable"\footnote{Computable means we can evaluate the function and its first and second derivatives at a reasonable cost.} s.c.\ barrier may not be available. However, as discussed in \cite{karimi_arxiv}, many problems that arise in practice can be handled by the Domain-Driven formulation. The examples presented in \cite{karimi_arxiv} are optimization over (1) symmetric cones (LP, SOCP, and SDP), (2) direct sums of an arbitrary collection of 2-dimensional convex sets defined as the epigraphs of univariate convex functions (including as special cases, geometric programming \cite{boyd2007tutorial} and entropy programming), (3) epigraph of relative entropy and vector relative entropy, (4) epigraph of a matrix norm (including as a special case, minimization of nuclear norm over a linear subspace), (5) epigraph of quantum entropy, and (6) any combination of all the above examples. 

\subsection{Contributions of this paper}  
The first contribution of this paper is classifying the possible statuses for convex optimization problems in a Domain-Driven form as in Table \ref{tbl:statuses}.
\begin{table}  [h]
\centering
  \caption{Possible statuses for a problem in Domain-Driven form. }
  \label{tbl:statuses}
  \begin{tabular}{ |c | c | }
    \hline
     {\bf Infeasible}  & {\bf Feasible} \\ \hline
     \begin{minipage}[t]{0.5\textwidth}
     \vspace{.2cm}
     \begin{itemize}
\item Strongly infeasible \begin{itemize}
\item Strictly infeasible
\end{itemize}
\item Ill-posed
\end{itemize}
\end{minipage}
 & \begin{minipage}[t]{0.5\textwidth}
     \begin{itemize}
\item Strictly primal-dual feasible 
\item Strongly Unbounded \begin{itemize}
\item Strictly Unbounded 
\end{itemize}
\item Ill-posed
 \vspace{.2cm}
\end{itemize}
\end{minipage}
    \\ \hline 
  \end{tabular}
\end{table}
Then, we study the geometric properties of the problem in different statuses. In this part, we exploit some properties of the Legendre-Fenchel (LF) conjugates of s.c.\ barriers, which are more than an arbitrary s.c.\ function \cite{karimi_arxiv}. 

 Then we focus on the polynomial time infeasible-start path following algorithm PtPCA  designed in \cite{karimi_arxiv,karimi_thesis} (a summary of the results we need come in Section \ref{sec:summary}) and discuss how the output of this algorithm can be interpreted to determine the status of a given problem. 
We discuss the certificates this algorithm returns (heavily relying on duality theory) for each of these cases, and analyze the number of iterations required to return such certificates.  Our approach (and in general interior-point methods) returns more robust certificates in provably stronger (polynomial) iteration complexity bounds compared to first-order methods such as Douglas-Rachford splitting \cite{liu2017new}, at the price  of higher computational cost per iteration. However, as explained in \cite{karimi_arxiv}, the quasi-Newton type ideas for deriving suitable primal-dual local metrics in \cite{tunccel2001generalization, myklebust2014interior} can be used to make our algorithm scalable, while preserving some primal-dual symmetry. The rest of the article covers:
\begin{itemize}
\item Discussing the strictly primal and dual feasible case and the more general case where there exists a pair of primal-dual feasible points with zero duality gap,  and proving that the PtPAC algorithm returns an approximate solution (with a certificate) in polynomially  many iterations (Section \ref{sec:solvable}). 
\item Defining a weak detector, which returns $\epsilon$-certificates of infeasibility or unboundedness in polynomial time  (Section \ref{sec:weak}). 
\item Defining a strict detector, which returns exact certificates when the problem is strictly infeasible or strictly unbounded  (Section \ref{sec:strict}). 
\item Studying the performance of PtPCA algorithm for some ill-conditioned cases (Section \ref{sec:ill}).
\item Designing the stopping criteria of the PtPAC algorithm for practice based on the insights gained from the analyses of the statuses (Section \ref{sec:conclude}).
\end{itemize}
 Our iteration complexity results are comparable with the current best theoretical iteration complexity bounds for conic formulations (mostly  given in \cite{infea-2}), and are new for the infeasible-start models used, even in the very special case of LP. The algorithms designed in \cite{karimi_arxiv} together with the output analysis results of this article make up the foundation of new software DDS (Domain-Driven Solver) for convex optimization problems.

\subsection{Notations and assumptions}

As justified in \cite{karimi_arxiv}, we assume that the kernel of $A$ in \eqref{main-p} is $\{0\}$ and also the Legendre-Fenchel (LF) conjugate $\Phi_*$ of $\Phi$ is given.
The domain of $\Phi_*$ is the interior of a cone $D_*$ defined as:
 \begin{eqnarray} \label{eq:leg-conj-2} 
D_*:=\{y:  \langle y , h \rangle  \leq 0, \ \ \forall h \in \rec(D)\}, 
\end{eqnarray}
where $\rec(D)$ is the recession cone of $D$.
Consider an Euclidean vector space $\mathbb E$ with dual space $\mathbb E^*$ and a scalar product $\langle \cdot, \cdot \rangle$.  For a self-adjoint positive definite linear transformation $B: \mathbb E \rightarrow \mathbb E^*$, we define a conjugate pair of Euclidean norms as:
\begin{eqnarray} \label{norms}
\|x\|_B &:=& \left [ \langle Bx,x \rangle \right] ^{1/2},  \nonumber \\
\|s\|^*_B &:=& \max\{\langle s,y \rangle: \ \|y\|_{B} \leq 1\} = \|s\|_{B^{-1}}=  \left [ \langle s, B^{-1}s \rangle \right] ^{1/2}.
\end{eqnarray} 
By using this definition, we have a general Cauchy-Schwarz  inequality:
\begin{eqnarray} \label{eq:CS}
\langle s,x \rangle  \leq \|x\|_B \|s\|^*_B, \ \ \forall x \in \mathbb E, \forall s \in \mathbb E^*. 
\end{eqnarray}
The abbreviations  RHS and LHS stand for right-hand-side and left-hand-side, respectively.

\section{Review of infeasible-start approaches for LP and general conic optimization} \label{sec:review}
Several infeasible-start interior-point approaches have been considered for LP and many of them have been extended to general convex optimization. 
In this section, we review some of these approaches and their iteration complexity. 
Having a feasible-start algorithm, an obvious approach for handling infeasibility is using a two-phase method.  
This approach is not desirable in practice and many researchers and practitioners are interested in approaches that solve the problem in a single phase. 
Another popular approach is the big-M approach where we reformulate our problem by adding some ``big" constants in a way that solving the reformulation lets us solve the initial problem. Assume that we want to solve the LP problem $\min \{c^\top x : Ax=b, x \geq 0\}$, where $A \in \R^{m \times n}$. When the data $(A,b,c)$ are rational, let $L$ be the size of the given data in the LP (the number of bits required to store the given data). The big-M approach has been used in interior-point methods \cite{kojima1989polynomial, monteiro1989interior, kojima1993little, kojima1989primal, mcshane1989implementation,megiddo1989pathways}  to achieve $O(\sqrt{n} L)$ number of iterations for solving the problem.
 This approach has major practical issues: (1) It is not clear how large the constants must be chosen (constants that are provably large enough for good theoretical behavior are typically unnecessarily large in practice), and (2) Introducing very large artificial constants to data tends to make the problem, and/or linear systems that arise in computations, ill-conditioned.

An elegant way of designing and analyzing interior-point algorithms involve \emph{potential functions} (which can be used to measure the progress of the algorithm, to find good search directions, and to find good step sizes). The underlying family of algorithms are called \emph{potential reduction algorithms}. Mizuno, Kojima, and Todd designed an infeasible-start potential reduction algorithm \cite{mizuno1995infeasible} for LP. Their purely potential reduction algorithm achieves  $O(n^{2.5}L)$ iteration complexity bound and the bound can be improved to $O(nL)$ by adding some centering steps. Seifi and Tun\c{c}el  \cite{seifi1998constant} designed another infeasible-start potential reduction algorithm with iteration complexity bound $O(n^{2}L)$. 

As explained in \cite{karimi_arxiv}, our infeasible-start approach is in the middle of two scenarios based on the number of artificial variables. In the scenario closer to ours (see \cite{lustig1990feasibility,lustig1991computational,kojima1993primal,zhang1994convergence,zhang1998extending}), the systems we solve at every iteration are the same as the ones we solve in the feasible-start case except for a perturbed RHS, and there is no artificial variable in the formulation. These algorithms have been very popular since late 1980's \cite{lustig1991computational,lipsol} and been recently used for even non-convex infeasible-start setups \cite{hinder-1,hinder-2}. However, their complexity analysis has been challenging \cite{kojima1993primal,zhang1994convergence, mizuno1994polynomiality, todd2003detecting}, and  in the case of LP the best bound for some variations of the approach is $O(nL)$ \cite{mizuno1994polynomiality}. At the other extreme are the infeasible-start algorithms which form  a homogeneous self-dual embedding \cite{YTM,infea-2} by adding artificial variables and homogenization variables.  Using this formulation, Ye, Todd, and Mizuno \cite{YTM} achieved the $O(\sqrt{n}L)$ iteration complexity bound for LP.  If we use a feasible-start algorithm that returns a strictly (self-)complementary solution, we can immediately solve both of the primal and dual problems \cite{YTM}. Many algorithms based on the homogeneous self-dual embedding formulation have been designed and implemented, see for example \cite{xu1996simplified}.
In our infeasible-start approach, we do not impose an explicit homogenization and add only one artificial variable which is tied to the central path parameter. Our complexity results here are new for this approach, even in the case of LP where our iteration bound is $O(\sqrt{n}L)$.

Let us shift our focus from LP to conic optimization problems.  Let $K \subset \R^n$ be a pointed closed convex cone, $\hat A: \R^n \rightarrow \R^m$ be a linear embedding, and $\hat c \in \R^n$ and $\hat b \in \R^m$ be given.  We define a primal-dual conic optimization pair as\footnote{ We use a hat for the data and parameters in the conic formulation as $\hat c, \hat \tau, \ldots$  and keep $c, \tau, \ldots$ for the Domain-Driven form.}:
\begin{eqnarray} \label{intro-1}
\text{(P)} &\inf& \{ \langle \hat c , \hat z \rangle :  \hat A \hat z= \hat b , \hat z \in K\}, \nonumber \\
\text{(D)} &\inf& \{ \langle \hat b , \hat v \rangle :  \hat s:=\hat c+\hat A^\top \hat v \in K^* \}, 
\end{eqnarray}
where $K^*:=\{\hat s: \langle \hat s, \hat z\rangle \geq 0, \forall \hat z \in K\}$ is the dual cone of $K$. We consider the infeasible-start approach of Nesterov \cite{infea-1} and its generalized version by Nesterov, Todd, and Ye \cite{infea-2} which uses a homogeneous self-dual embedding. \cite{infea-2}, as far as we know, is the most comprehensive result for infeasible-start interior-point methods for conic optimization to date and we compare our results with it. For arbitrary starting points $\hat z^0 \in \inte K$, $\hat s^0 \in \inte K^*$, $\hat v^0 \in \R^m$, and $\hat \tau_0, \hat \kappa_0 >0$, we define
\begin{eqnarray}  \label{eq:Q-A}
\begin{array}{rcl}
Q &:=&\left \{ ( \hat z,\hat \tau, \hat s, \hat v, \hat \kappa): \  \hat A \hat z= \check b +\hat \tau \hat b, \ \ \hat s=\check c+\hat \tau \hat c + \hat A^\top \hat v,  \right. \\
  && \ \ \ \left. \langle \hat c, \hat z \rangle+\langle \hat b,\hat v \rangle + \hat \kappa= \check g, \ \  \hat z \in \inte K, \ \hat s \in \inte K^*, \ \hat \tau, \hat \kappa >0 \right \},
\end{array}
\end{eqnarray}
where $\check b:= \hat A \hat z^0-\hat \tau_0 \hat b$, $\check c:= -\hat A^\top \hat v^0+\hat s^0-\hat \tau_0 \hat c$, and $\check g:= \langle \hat c, \hat z^0 \rangle + \langle \hat b,\hat v^0 \rangle +\hat \kappa_0$. 
The authors in \cite{infea-2} solved \eqref{intro-1} by finding a recession direction for $Q$. 
Note that $\langle \hat c,\hat z \rangle + \langle \hat b,\hat v \rangle$ is the conic duality gap. Assume that we have a point in $Q$ with a large $\hat \tau >0$. Then, $(\hat z/\hat \tau,\hat s/ \hat\tau)$ approximately satisfies all the optimality conditions, and if $\hat \tau$ tends to infinity, it converges to a primal-dual optimal solution. Similar principles underlie our approach. 
\section{Some definitions and results about the Domain-Driven formulation} \label{sec:summary}
In this section, we summarize the results we need from \cite{karimi_arxiv} including the definition of the duality gap, the primal-dual central path, and the theorem that shows there exists an algorithm PtPCA that follows the path efficiently. 

For every point $ x \in \R^n$ such that $A  x \in D$ and every point $ y \in D_*$ such that $A^\top y = -c$, \emph{the duality gap} is defined as:
\begin{eqnarray} \label{eq:duality-gap-1}
\langle c, x \rangle + \delta_*(y|D),
\end{eqnarray}
where
\begin{eqnarray}  \label{eq:supp-fun-1}
\delta_*(y|D) := \sup\{ \langle y,z \rangle :  z \in D\},  \ \ \ \text{(support function of $D$).}
\end{eqnarray}
Lemma 2.1 in \cite{karimi_arxiv} shows that duality gap is well-defined and zero duality gap is a guarantee for optimality.
Let us fix an absolute constant $\xi > 1$ and define the initial points:
\begin{eqnarray} \label{starting-points-copy-2}
z^0 := \text{any vector in  $\inte(D)$}, \ \ y^0 :=  \Phi'(z^0), \ \ y_{\tau,0} := -\langle y^0, z^0 \rangle -\xi \vartheta.
\end{eqnarray} 
Then, it is proved in \cite{karimi_arxiv} that the system 
\begin{eqnarray} \label{trans-dd-path-1-copy-2}
\begin{array}{rcl}
&(a)&  A x + \frac{1}{\tau} z^0 \in \inte D, \ \ \tau > 0,  \\
&(b)& A^\top y -A^\top y^0 = -(\tau-1) c, \ \ y \in \inte D_*,  \\
&(c)& y=\frac{\mu  }{\tau}  \Phi' \left (  A x + \frac{1}{\tau} z^0 \right), \\
&(d)& \langle c,x \rangle +  \frac{1}{\tau} \langle y, Ax+\frac{1}{\tau} z^0 \rangle =  -\frac{\vartheta \xi \mu}{\tau^2} +\frac{ -y_{\tau,0}}{\tau},
\end{array}
\end{eqnarray}
has a unique solution $(x(\mu), \tau(\mu), y(\mu))$ for every $\mu > 0$. The solution set of this system for all $\mu >0$ defines our \emph{infeasible-start primal-dual central path}. 
 Let us give a name to the set of points that satisfy \eqref{trans-dd-path-1-copy-2}-(a)-(b):
\begin{eqnarray} \label{QDD-copy-2}
 \ \ \ \ \ Q_{DD}:= \left \{ (x,\tau,y): A x + \frac{1}{\tau} z^0 \in \inte D, \ \tau > 0, \ \  A^\top y -A^\top y^0 = -(\tau-1) c, \   y \in \inte D_*  \right \}. 
\end{eqnarray}
In view of the definition of the central path, for all the points $(x,\tau,y)\in Q_{DD}$, we define
\begin{eqnarray}  \label{eq:dd-4-2}
\begin{array}{rcl}
\mu(x,\tau,y) &:=&   \frac{\tau}{\xi \vartheta}[-y_{\tau,0} - \tau \langle c, x \rangle - \langle y, A x + \frac{1}{\tau} z^0 \rangle], \\
                     &=&- \frac{1}{\xi \vartheta} \left [ \langle y,z^0 \rangle +\tau (y_{\tau,0}+\langle y, Ax \rangle) + \tau^2 \langle c,x \rangle \right] \\
                     &=& - \frac{1}{\xi \vartheta} \left [ \langle y,z^0 \rangle +\tau (y_{\tau,0}+\langle A^\top y^0+c,x \rangle) \right], \ \ \ \text{using \eqref{trans-dd-path-1-copy-2}-(b).}
                     \end{array}
\end{eqnarray}
We say that a point $(x,\tau,y) \in Q_{DD}$ is \emph{$\kappa$-close to the central path} if
\begin{eqnarray} \label{eq:lem:dg-bound-1-2}
\left\|Ax+\frac{1}{\tau} z^0-\Phi'_*\left(\frac{\tau}{\mu} y\right)\right\|_{[\Phi''_*(\frac{\tau}{\mu} y)]^{-1}} \leq \kappa,
\end{eqnarray}
where $\mu:=\mu(x,\tau,y)$.
In the rest of this article, for the analysis of the algorithms, we assume that the neighborhoods of the central path are chosen such that $\xi-1-\kappa > 0$. 
For the points $\kappa$-close to the central path, we have a bound on the duality gap as follows. 
\begin{lemma}  \label{lem:dg-bound-1}
Let $(x,\tau,y) \in Q_{DD}$ be $\kappa$-close to the central path and $\mu:=\mu(x,\tau,y)$. Then,
\begin{eqnarray} \label{eq:dg-bound-1}
-\left(\frac{y_{\tau,0}}{\tau}+\frac{\xi\mu\vartheta}{\tau^2}\right) -\kappa \frac{\mu \sqrt{\vartheta}}{\tau^2} \leq \langle c,x \rangle +\frac{1}{\tau} \delta_*\left(y|D\right)  \leq -\left(\frac{y_{\tau,0}}{\tau}+\frac{\xi\mu\vartheta}{\tau^2}\right) +\kappa \frac{\mu \sqrt{\vartheta}}{\tau^2}+ \frac{\mu \vartheta}{\tau^2}.
\end{eqnarray}
\end{lemma}
A \emph{polynomial-time predictor-corrector algorithm} (PtPCA) is designed in \cite{karimi_arxiv} that follows the path efficiently in the following sense:
\begin{theorem} \label{thm:complexity-result}
For the polynomial-time predictor-corrector algorithm (PtPCA), there exists a positive absolute constant $\gamma$ depending on $\xi$  such that after $N$ iterations, the algorithm returns a point $(x,\tau,y) \in Q_{DD}$ close to the central path that satisfies
\begin{eqnarray} \label{eq:thm-com}
\mu(x,\tau,y)  \geq \exp\left(\frac{\gamma}{\sqrt{\vartheta}} N \right). 
\end{eqnarray}
\end{theorem}

At this stage, we assume that the algorithm returns a point $(x,\tau,y)$, $\kappa$-close to the central path, with a large enough $\mu(x,\tau,y)$. We need to interpret such a point to classify the status of a given instance as accurately as possible.  In this paper, we first categorize  the possible statuses for a problem instance in the Domain-Driven setup. Then we discuss the statuses that we can determine and their complexity analysis for obtaining the corresponding certificates. 
\section{Categorizing problem statuses} \label{sec:cat-statuses}
 Let us first define the following four parameters that are the measurements of primal and dual feasibility:
\begin{definition} \label{tf}
For a linear embedding $x \mapsto Ax$ : $\R^n \rightarrow \R^m$, a closed convex set $D \subset \mathbb \R^m$ with a nonempty interior, and $D_*$ defined in \eqref{eq:leg-conj-2}, we define ($\range(A)$ is the range of $A$)
\begin{eqnarray} \label{t'}
\sigma_p &:=&\dist(\range(A), D), \nonumber \\
\sigma_d &:=& \dist(\{y: A^\top y=-c\}, D_*),
\end{eqnarray}
where $\text{dist}(\cdot,\cdot)$ returns the distance between two convex sets. We call $\sigma_p$ the \emph{measure of primal infeasibility}, and $\sigma_d$ the \emph{measure of dual infeasibility}. For $z^0 \in \inte D$ and $y^0 \in \inte D_*$, we define
\begin{eqnarray} \label{t'-2}
t_p(z^0) &:=&\sup \{t \geq 1 : \exists x \in \R^n \ \text{s.t.} \ Ax+\frac{1}{t} z^0 \in D\}, \nonumber \\
&=&\sup \{t \geq 1 : \exists x \in \R^n \ \text{s.t.} \ Ax+ z^0 \in tD\}, \nonumber \\
t_d(y^0) &:=&\sup \{t \geq 1: \exists y \in D_* \ \text{s.t.} \ A^\top y-A^\top y^0=-(t-1)c \}. 	
\end{eqnarray}
\end{definition}
Note that all the above measures are scale dependent. For example, $t_p(z^0)$ attains different values when we change $z^0$ with respect to the boundary of the set $D$. 
The following lemma connects the parameters defined in Definition \ref{tf}. 
\begin{lemma}  \label{lem:dd-con-par}
Let $x \mapsto Ax : \R^n \rightarrow \R^m$ be a linear embedding, $c \in \mathbb \R^n$, and $D \subset \mathbb \R^m$ be a closed convex set with a nonempty interior. Then, \\
(a) for every $z^0 \in \inte D$ we have
\begin{eqnarray} \label{eq:lem:dd-con-par-1}
\sigma_p  \leq \frac{\|z^0\|}{t_p(z^0)};
\end{eqnarray}
(b) for every $y^0 \in \inte D_*$ we have
\begin{eqnarray} \label{eq:lem:dd-con-par-1-2}
\sigma_d  \leq \frac{\|y^0\|}{t_d(y^0)}.
\end{eqnarray}
\end{lemma}
\begin{proof}
Let $z^0 \in \inte D$ be arbitrary and let sequences $\{x^k\} \subset \R^n$ and $\{t_k\} \subset [1,+\infty)$ be such that $Ax^k+\frac{1}{t_k} z^0 \in D$ and $\lim_{k \rightarrow +\infty} t_k =t_p(z^0)$.  Then, by definition of $\sigma_p$, for every $k$, we have
	\begin{eqnarray*}
		\sigma_p \leq \left\| Ax^k+\frac{1}{t_k} z^0-Ax^k \right\|=\frac{\|z^0\|}{t_k}.
	\end{eqnarray*}
We obtain \eqref{eq:lem:dd-con-par-1} when $k$ tends to $+\infty$. Part (b) can be proved similarly. 
\end{proof}
\noindent Here is the classification of statuses for problem \eqref{main-p}. For simplicity, we define 
\begin{eqnarray} \label{eq:feas-regions}
\mathcal F_p:=\{x \in \R^n: Ax \in D\}, \ \ \ \mathcal F_d:=\{y \in D_*: A^\top y=-c\},
\end{eqnarray}
and the duality gap is defined as
\[
\Lambda:= \inf_{x,y} \{\langle c, x\rangle +\delta_*(y|D) : x \in \mathcal F_p, y \in \mathcal F_d\}.
\]
\noindent {\bf Infeasible:} Problem \eqref{main-p} is called \emph{infeasible} if $\mathcal F_p$ is empty.
\begin{enumerate}[(i)]
\item {\bf Strongly infeasible:} there exists $y \in D_*$ such that $A^\top y=0$ and $\delta_*(y|D)=-1$ (equivalent to $\sigma_p >0$ by Lemma \ref{lem:outcome-infeas}). 

 {\bf Strictly infeasible:} there exists such a $y$ in $\inte D_*$. 

\item {\bf Ill-posed:} problem is infeasible, but $\sigma_p =0$. 
\end{enumerate}
\noindent {\bf Feasible:} Problem \eqref{main-p} is called \emph{feasible} if $\mathcal F_p$ is nonempty.
\begin{enumerate}[(i)]
	\item
{\bf Strictly primal-dual feasible:} there exist $ x$ such that $A  x \in \inte D$ and $ y \in \inte D_*$ such that $A^\top y = -c$.

\item {\bf Strongly unbounded:} there exist $x \in \R^n$ with $Ax \in \inte D$ and $h \in \R^n$ with $Ah \in \rec(D)$ such that $\langle c,h \rangle <0$. (equivalent to $\sigma_d >0$ by Lemma \ref{lem:outcome-infeas}). 

{\bf Strictly unbounded:} there exists such an $h$ with $Ah \in \inte(\rec(D))$. 
	
\item \textbf{Ill-posed:} problem is not strictly primal-dual feasible or strongly unbounded:
\begin{enumerate}
\item \textbf{Unstable Optima (with dual certificate):} there exist $x \in \mathcal F_p$ and $ y \in \mathcal F_d$ with duality gap equal to zero (i.e., $\langle c, x \rangle + \delta_*(y|D)=0$). 
\item \textbf{Unstable Optima (without dual certificate):} primal problem has an optimal solution $x^*$, however
\begin{itemize}
\item the set $\mathcal F_d$ is empty.
\item $\mathcal F_d$ is nonempty, and the duality gap is zero, but there does not exist $y \in \mathcal F_d$ such that $\langle c, x^* \rangle + \delta_*(y|D)=0$. 
\item $\mathcal F_d$ is nonempty, the duality gap is $\Lambda >0$, and there exists $y \in \mathcal F_d$ such that $\langle c, x^* \rangle + \delta_*(y|D)=\Lambda$.
\item $\mathcal F_d$ is nonempty, the duality gap is $\Lambda >0$, but there does not exist $y \in \mathcal F_d$ such that $\langle c, x^* \rangle + \delta_*(y|D)=\Lambda$.
\end{itemize}
\item {\bf Unsolvable:}  $\mathcal F_p$ is nonempty and the objective value of the primal problem is bounded, but the primal optimal value is not attained. In the dual side, 
\begin{itemize}
\item the set $\mathcal F_d$ is empty.
\item $\mathcal F_d$ is nonempty, and the duality gap is zero, but there does not exist $y \in \mathcal F_d$ such that $\inf_{x\in \mathcal F_p} \{\langle c, x \rangle + \delta_*(y|D)\}=0$. 
\item $\mathcal F_d$ is nonempty, the duality gap is $\Lambda >0$, and there exists $y \in \mathcal F_d$ such that $\inf_{x\in \mathcal F_p} \{\langle c, x \rangle + \delta_*(y|D)\}=\Lambda$.
\item $\mathcal F_d$ is nonempty, the duality gap is $\Lambda >0$, but there does not exist $y \in \mathcal F_d$ such that $\inf_{x\in \mathcal F_p} \{\langle c, x \rangle + \delta_*(y|D)\}=\Lambda$.
\end{itemize}
\end{enumerate}
\end{enumerate}

The following lemma connects strong infeasibility and unboundedness  to $\sigma_p$ and $\sigma_d$. 
\begin{lemma}  \label{lem:outcome-infeas}
Let $x \mapsto Ax :  \R^n \rightarrow \R^m$ be a linear embedding with kernel $\{0\}$, $D \subset \mathbb \R^m$ be a closed convex set with a nonempty interior, and $D_*$ be defined as in \eqref{eq:leg-conj-2}. \\
\noindent (a) There exists $y \in D_*$ such that $A^\top y=0$ and $\delta_*(y|D)=-1$ if and only if 
\begin{eqnarray} \label{eq:dd-outcome-2-2}
\sigma_p = \dist( \textup{range} (A), D)  >0.
\end{eqnarray}
\noindent (b) Assume that $\{x \in \R^n: Ax \in D\}$ is nonempty. For a vector $c \in \R^n$, there exists $h \in \R^n$ with $Ah \in \rec(D)$ such that $\langle c,h \rangle <0$ if and only if 
\begin{eqnarray} \label{eq:dd-outcome-2-3}
\sigma_d=\dist(\{y: A^\top y=-c\}, D_*)  >0.
\end{eqnarray}
		
\end{lemma}
\begin{proof}
(a) First assume that there exists $y \in D_*$ such that $A^\top y=0$ and $\delta_*(y|D)=-1$. Consider two sequences $\{z^k\} \subset D$ and $\{x^k\} \subset \R^n$ such that 
\[
\lim_{k \rightarrow +\infty}\left\|z^k-Ax^k\right\|= \text{dist} (D, \text{range} (A)).
\]
Then, we have 
\begin{eqnarray} \label{eq:lem:outcome-infeas}
\langle y, z^k-Ax^k \rangle = \langle y, z^k \rangle \leq \delta_*(y|D) =-1, \ \ \ \forall k \in \Z_+. 
\end{eqnarray}
Using \eqref{eq:CS}, we have $1 \leq |\langle y, z^k-Ax^k \rangle | \leq \left\|z^k-Ax^k\right\| \|y\|$ which implies that $ \text{dist} (D, \text{range} (A)) >0$. For the other direction, assume that $ \text{dist} (D, \text{range} (A))$ is positive. Then, by a separating hyperplane theorem applied to nonempty, closed convex sets $D$ and $\range(A)$, there exist  $y \in \R^m$ and $\beta \in \R$ such that 
\[
\begin{array}{ll}
\langle y, z \rangle > \beta, &  \forall z \in \range(A), \\
\langle y, z \rangle < \beta,  &  \forall z \in D. 
\end{array}
\]
The first relation holds only if $A^\top y =0$, and if we substitute $z=0$ in it, we get $\beta <0$. The second relation holds only if $y \in D_*$ by the definition of $D_*$ in \eqref{eq:leg-conj-2}, and since $\beta < 0$, we have $\delta_*(y|D) < 0$. Suitably scaling $y$,  we may assume that $\delta_*(y|D) = -1$.\\
\noindent (b) Assume that there exists $h \in \R^n$ such that  $Ah \in \rec(D)$ and  $\langle c, h \rangle < 0$. If $\sigma_d=0$, there exists $\{y^k\}\subset D_*$ such that $\lim_{k\rightarrow +\infty} \|A^\top y^k +c\| =0$. By characterization of $D_*$ in \eqref{eq:leg-conj-2}, we have 
\[
0 \geq \langle y^k, Ah\rangle = \langle A^ \top y^k, h\rangle =  \langle A^ \top y^k+c, h\rangle -  \langle c, h\rangle, \ \ \forall k. 
\]
This gives a contradiction when $k$ tends to $+\infty$. For the other direction, assume that $\sigma_d >0$. As kernel of $A$ is $\{0\}$, the set $\{y: A^\top y=-c\}$ is nonempty. Similar to part (a), there exist  $z \in \R^m$ and $\beta \in \R$ such that 
\[
\begin{array}{ll}
\langle y, z \rangle > \beta, &  \forall y \in \{y: A^\top y=-c\}, \\
\langle y, z \rangle < \beta,  &  \forall y \in D_*. 
\end{array}
\]
As $D_*$ is a cone, we must have $\beta \geq 0$ and $\langle y,z \rangle \leq 0$ for every $y \in D_*$; by the definition of $D_*$ in \eqref{eq:leg-conj-2} we have $z \in \rec(D)$. Let us write $z$ as $Ah+g$ for $h \in \R^n$ and $g \in \R^m$ in the kernel of $A^\top$. We claim that $g=0$. Let $y_c$ be any vector such that $A^\top y_c =-c$ (since kernel of $A$ is $\{0\}$, such a vector exists). Then for every $\alpha >0$ we also have $A^\top (y_c\pm\alpha g)=-c $, which implies that 
\[
\beta < \langle y_c\pm\alpha g, z\rangle = \langle y_c\pm \alpha g, Ah+g\rangle=-\langle c,h \rangle+\langle y_c,g\rangle \pm \alpha \langle g,g\rangle,
\] 
for every $\alpha >0$. Therefore, $g=0$ and we have $\langle c, h \rangle < 0$. 
\end{proof}
Let us see an example to elaborate more on ill-posed cases. 
\begin{example}
Define a convex set $D:=\{(x_1,x_2)^\top \in \R^2: x_1 \geq \frac{1}{x_2}, 0 \leq x_2 \leq 2\}$, shown in Figure \ref{Fig8}, and $A:=I_{2\times 2}$, the identity matrix. Note that the recession cone of $D$ is a ray, which implies that $D_*$ defined in \eqref{eq:leg-conj-2} is $\{(y_1,y_2)^\top \in \R^2: y_1 \leq 0 \}$ as shown in Figure \ref{Fig8}. Let us define $c:=[0 \ 1]^\top$; then, the optimal objective value of the primal is 0 which is not attained. The system $A^\top y=-c$ has a unique solution $\bar y:=[0 \ -1]^\top$ that is on the boundary of $D_*$. It is also clear from the figure that $\delta_*(\bar y|D) \leq 0$. Therefore, both primal and dual problems are feasible; however, we do not have a pair of primal-dual feasible solutions with zero duality gap. 
\begin{figure}[h]
	\centering
	\includegraphics[scale=0.5] {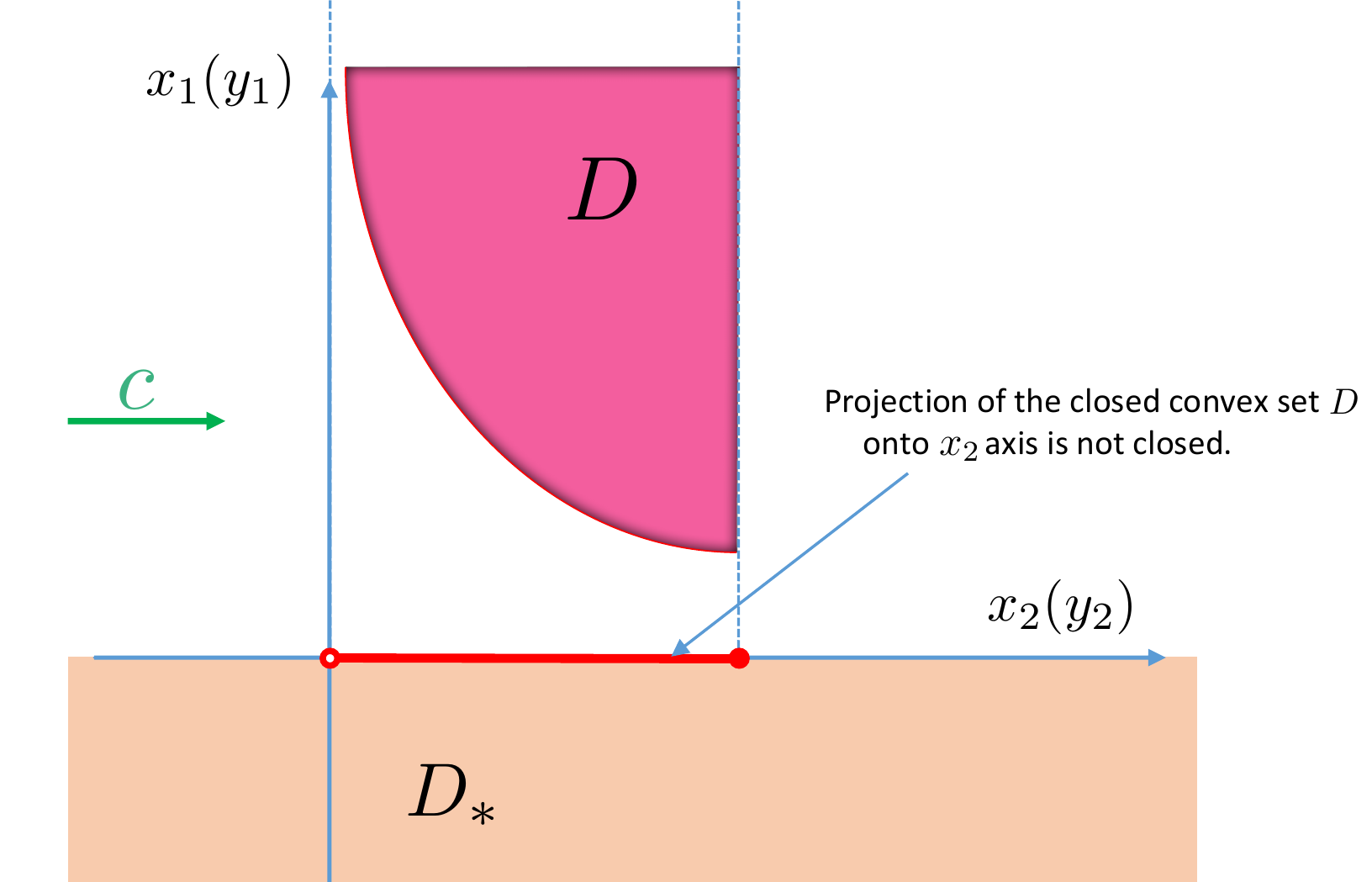}
	\caption{ \small An example of problem \eqref{main-p} with $D \subset \R^2$. }
	\label{Fig8}
\end{figure}

Now assume that we change $A$ to $A:=[1 \ 0]^\top$ ($\range(A)=(\R,0)$). It is clear from the picture that the feasible region of problem \eqref{main-p} is empty. If we choose $c=-1$, then the system $A^\top y= y_1=-c$ does not have a solution in $D_*$. This implies that both primal and dual are infeasible. However, the measure of primal feasibility $\sigma_p$ we defined in \eqref{t'} is zero and we have approximately feasible points with arbitrarily small objective values. If we shift $D$ to the right, $D_*$ does not change, but we can make $\sigma_p$ arbitrarily large. 
\end{example}
Table \ref{table:comparison}  compares the complexity bounds we derive in this paper and the corresponding ones for the conic setup in \cite{infea-2}. 
\begin{center}
\begin{table} [h]
\caption{\small Comparison of complexity bounds of Domain-Driven and conic setup for different statuses. The bounds for the strictly infeasible and unbounded cases recover the conic ones in case $D$ is a closed convex cone.}
\label{table:comparison}
\renewcommand{\arraystretch}{1.5}
\begin{tabular}{|c|c|c|}
\hline
 {\bf Status} & {\bf Domain-Driven form} &    {\bf Conic form \cite{infea-2}} \\ \hline
 \multicolumn{1}{l}{\emph{Strict Detectors}} \\ \hline
Strictly &  $O\left(\sqrt{\vartheta} \ln \left(\frac{\vartheta}{\sigma_f \epsilon}\right)\right)$   &   $O\left(\sqrt{\hat\vartheta} \ln \left(\frac{\hat \vartheta}{\rho_f \epsilon}\right)\right)$ \\ 
  primal-dual feasible  & $\sigma_f=$ feasibility measure & $\rho_f=$ feasibility measure \\ \hline
  Strictly &  $O\left(\sqrt{\vartheta} \ln \left(t_p(z^0) B_{p,\tau_{\xi,\kappa}}\right)\right)$   &   $O\left(\sqrt{\hat\vartheta} \ln \left(\frac{\hat \vartheta}{\rho_p }\right)\right)$ \\ 
  infeasible  & $B_{p,\tau_{\xi,\kappa}}=$ a bound on $x$ vectors & $\rho_p=$ primal infeasibility measure \\ \hline
  Strictly  unbounded &  $O \left ( \sqrt{\vartheta}  \ln \left (    \frac{1}{\vartheta\epsilon} t_d(y^0)+ t_d(y^0) B_{d,0} \right )\right)$   &   $O\left(\sqrt{\hat\vartheta} \ln \left(\frac{\hat \vartheta}{\rho_d}\right)\right)$ \\ 
(dual infeasible) & $B_{d,0}=$ a bound on $y$ vectors & $\rho_d=$ dual infeasibility measure \\ \hline
 \multicolumn{1}{l}{\emph{Weak Detectors}}   \\ \hline
 Unstable Optima  &  $O\left(\sqrt{\vartheta} \ln \left(\frac{\vartheta B}{\epsilon}\right)\right)$   &   $O\left(\sqrt{\hat \vartheta} \ln \left(\frac{\hat \vartheta \hat B}{\epsilon}\right)\right)$ \\ 
(with dual  & $B=\xi + \min \{ \frac{1}{\vartheta}\langle y^0- \bar y, z^0-A \bar x  \rangle: (\bar x,\bar y)$ &  $\hat B=1+\min \{ \langle \hat s^0, \hat z \rangle+ \langle \hat s, \hat z^0   \rangle: (\hat  z,\hat  s)$  \\ 
certificate) &      $ \text{primal-dual feas with 0 duality gap} \}$             &         $ \text{primal-dual feas with 0 duality gap} \}$       \\ \hline
  Weak detector for &  \multirow{3}{*}{ $O\left(\sqrt{\vartheta} \ln \left( \frac{1}{\vartheta\epsilon}  \min \left\{ \frac{\|z^0\|}{\sigma_p},\frac{\|y^0\|}{\sigma_d} \right\} \right)\right)$ }  &   $O\left(\sqrt{\hat\vartheta} \ln \left(\frac{\hat \vartheta}{ \epsilon} \hat B
  \right)\right)$ \\ 
  infeasibility  &  & $\hat B=$ a function of primal/dual \\
 unboundedness &  & infeasibility certificates \\ \hline
\end{tabular}
\end{table}
\end{center}

\section{{Strict primal-dual feasibility}} \label{sec:solvable}
 Let us first review the closest related results from the literature for conic optimization. Assume that for the conic formulation \eqref{intro-1}, both primal and dual problems are strictly feasible, and let $\hat{\bar z} \in \inte K$ and $\hat{\bar s} \in \inte K^*$ such that $\hat{\bar s} = -(\Phi^+)'(\hat{\bar z})$ and $\hat \vartheta:=\langle \hat{\bar s},\hat{\bar z}\rangle$, where $\Phi^+$ is a $\hat \vartheta$-LH s.c.\ barrier defined on $K$ \cite{interior-book}. Then, the conic \emph{feasibility measure} $\rho_{f}$ is defined in \cite{infea-1,infea-2} as
\begin{eqnarray} \label{eq:conic-mea-1}
	\rho_f:= \max \left \{ \alpha: \ \   \hat{\bar z}-\alpha \hat z^0 \in K, \ \  \hat{\bar s}-\alpha \hat s^0 \in K^*  \right\}.
\end{eqnarray}
Theorem 9 in \cite{infea-2} shows that for every point in $Q$ defined in \eqref{eq:Q-A}
we have
\begin{eqnarray} \label{eq:conic-mea-2}
	\hat \tau  \geq \frac{\hat \vartheta +1}{\hat \vartheta + \rho_f} \rho_f \hat \mu - \frac{1-\rho_f}{\rho_f},
\end{eqnarray}
where $\hat \mu$ is a function defined similar to \eqref{eq:dd-4-2} for the conic formulation. 
This inequality is important as it shows how $\hat \tau$ is lower-bounded by an increasing linear function of $\hat \mu$.  \cite{infea-2} also considers a case called ``solvable" where there exists a primal-dual feasible pair $(\hat z^*, \hat s^*)$ with duality gap equal to zero.  Theorem 10 in \cite{infea-2} shows that for the points close to the central path we have
\begin{eqnarray} \label{eq:conic-mea-3}
\hat \tau  \geq \frac{\omega \hat\mu }{\langle \hat s^0, \hat z^* \rangle + \langle \hat s^*, \hat z^0 \rangle + 1},
\end{eqnarray}
where $\omega$ is a positive absolute constant regulating proximity to the central path. For the Domain-Driven form, we consider the case of strict primal-dual feasibility and also the case where strict feasibility fails, but there exists a primal-dual feasible pair with zero duality gap (which is an ill-posed status in our categorization). 
For the case of strict primal and dual feasibility (there exist $x \in \R^n$ such that  $A  x \in \inte D$ and $ y \in \inte D_*$ such that $A^\top y=-c$), let us define
\begin{eqnarray} \label{eq:dd-outcome-29}
	 \bar x(1) &:=& \argmin_x  \{\Phi(Ax)+\langle c,x \rangle \},  \nonumber \\
	\bar y(1) &:=& \Phi'(A \bar x(1)),  \nonumber \\
	\bar y_\tau(1) &:=&  -\xi \vartheta - \langle \bar y(1), A \bar x(1) \rangle.
\end{eqnarray}
$\bar x(1)$ is well-defined by \cite{cone-free}-Lemma 3.1. By using the first order optimality condition, we have $A^\top \bar y(1) = -c$. Now we define the \emph{feasibility measure} as
\begin{eqnarray*} \label{eq:dd-outcome-30}
\sigma_f := \sup \left\{\alpha: \alpha <  1, \ \bar y(1) - \alpha y^0  \in D_*,  \ \frac{A \bar x(1) - \alpha z^0}{ 1- \alpha}  \in D, \ \delta_* (\bar y(1) - \alpha y^0|D) +  \bar y_\tau(1) - \alpha y _{\tau,0}  \leq 0   \right\}.
\end{eqnarray*}
Note that using \cite{interior-book}-Theorem 2.4.2 and the fact that $\bar y(1) = \Phi'(A \bar x(1))$, we have
\[
\delta_*(\bar y(1)|D) + \bar y_\tau(1) \leq \vartheta + \langle \bar y(1), A \bar x(1) \rangle -\xi \vartheta - \langle \bar y(1), A \bar x(1) \rangle = -(\xi-1) \vartheta  < 0.
\]
Hence, $\sigma_f > 0$. The following theorem gives a result similar to \eqref{eq:conic-mea-2} for the Domain-Driven setup. 

\begin{theorem} \label{thm:fes-meas-dd}
	Assume that both primal and dual are strictly feasible. For every point $(x,\tau,y) \in Q_{DD}$ with the additional property that $\delta_*(y|D)+y_\tau \leq 0$, where $y_\tau:=y_{\tau,0}+\tau \langle c,x \rangle$, we have 
\begin{eqnarray} \label{eq:thm:fes-meas-dd-1}
\tau-1  \geq  \sigma_f \mu(x,\tau,y) - \frac{1}{\sigma_f}. 
\end{eqnarray}
\end{theorem}
\begin{proof}
For two points of the form $(\tilde z,\tilde\tau,\tilde y,\tilde y_\tau),(\bar z,\bar \tau,\bar y,\bar y_\tau) \in \R^m\oplus\R\oplus\R^m\oplus\R$, we use $\bullet$ to denote the natural scalar product in this context:
\[
(\tilde z,\tilde\tau,\tilde y,\tilde y_\tau)\bullet (\bar z,\bar \tau,\bar y,\bar y_\tau):= \langle \tilde y, \bar z \rangle + \bar \tau \tilde y_\tau+\langle \bar y, \tilde z \rangle + \tilde\tau \bar  y_\tau. 
\]
The proof is based on the three scalar products among three points $(\tau Ax+z^0, \tau, y, y_\tau)$, $(z^0,1,y^0, y_{\tau,0})$, and $(A\bar x(1),1,\bar y(1), \bar y_\tau(1))$. First we claim that 
\begin{eqnarray} \label{eq:dd-outcome-33}
-(A\bar x(1),1,\bar y(1), \bar y_\tau(1)) \bullet (z^0,1,y^0, y_{\tau,0})  \leq \xi \vartheta  \left ( \frac{1}{\sigma_f} + \sigma _f \right).
\end{eqnarray}
Consider a sequence $\{\alpha_k\} \subset (0,\sigma_f)$ such that $\lim_{k \rightarrow+\infty} \alpha_k = \sigma_f$.  By definition of $\sigma_f$ and $\delta_*$, for every $k$ we have
\begin{eqnarray*} \label{eq:dd-outcome-31}
\langle \bar y(1) - \alpha_k y^0, \frac{A \bar x(1) - \alpha_k z^0}{ 1- \alpha_k} \rangle  +  \bar y_\tau(1) - \alpha_k y _{\tau,0} \ \leq \ \delta_*(\bar y(1) - \alpha_k  y^0|D) + \bar y_\tau(1) - \alpha_k y _{\tau,0} \ \leq 0.
\end{eqnarray*}
Multiplying both sides with $(1-\alpha_k)$, reordering the terms, and taking the limit as $k \rightarrow +\infty$  give us
\begin{eqnarray*} \label{eq:dd-outcome-32}
\langle  \bar y(1) , A \bar x(1)  \rangle + \bar y_\tau(1) -  \sigma _f \left (\langle \bar y(1) , z^0  \rangle + \bar y_\tau(1) +\langle y^0 , A \bar x(1)  \rangle + y _{\tau,0} \right )  + \sigma_f ^2  \left ( \langle  y^0 ,  z^0  \rangle + y _{\tau,0} \right ) \leq 0.
\end{eqnarray*}
By \eqref{eq:dd-outcome-29} we have $\langle  \bar y(1) , A \bar x(1)  \rangle + \bar y_\tau(1) = -\xi \vartheta$ and by \eqref{starting-points-copy-2} we have $\langle  y^0 ,  z^0  \rangle + y _{\tau,0} = -\xi \vartheta$. Substituting these in the above inequality and dividing both sides by $\sigma_f$ we get \eqref{eq:dd-outcome-33}. The second claim is that
\begin{eqnarray*} \label{eq:dd-outcome-34}
-(\tau Ax+z^0, \tau, y, y_\tau)\bullet (A\bar x(1),1,\bar y(1), \bar y_\tau(1))  \geq \sigma_f (-(\tau Ax+z^0, \tau, y, y_\tau)\bullet (z^0,1,y^0, y_{\tau,0})).
\end{eqnarray*}
To prove this, we need the following two inequalities: 
\begin{eqnarray} \label{eq:inter-1}
\begin{array}{rcl}
\langle y, A \bar x(1) - \sigma _f z^0  \rangle + (1-\sigma_f) y_\tau &\leq& 0,  \\ 
\langle  \bar y(1)-\sigma_f y^0,  Ax+\frac{1}{\tau} z^0  \rangle + (\bar y_\tau(1) - \sigma_f y_{\tau,0})  &\leq& 0. 
\end{array}
\end{eqnarray}	
The first one is by using the hypothesis of the theorem and then the definition of $\delta_*$. The second one holds by using the definition of $\sigma_f$ and $\delta_*$ and a similar argument we made by using $\{\alpha_k\}$. 
If we multiply the second inequality in \eqref{eq:inter-1} by $-\tau$, add it to the negation of the first inequality, add $\sigma_f \left ( -\langle y,z^0  \rangle -  y_\tau - \langle  y^0,\tau Ax+z^0  \rangle - \tau  y_{\tau,0} \right)$ to both sides, and simplify, we prove the second claim.
For the next relation between the $\bullet$ products, by adding and subtracting $(\tau-1) \bar y_\tau(1)$ and using $\bar y_\tau(1) = -\xi \vartheta - \langle \bar y(1), A \bar x(1) \rangle=-\xi \vartheta +\langle c,\bar x(1) \rangle$, and also using $y _{\tau,0}=y_\tau-\tau \langle c,x \rangle$, for the LHS of \eqref{eq:dd-outcome-33} we have
\begin{eqnarray} \label{eq:dd-outcome-35}
\nonumber 
\begin{array}{rclc}
&& -(A\bar x(1),1,\bar y(1), \bar y_\tau(1)) \bullet (z^0,1,y^0, y_{\tau,0}) & \\
&=& (\tau-1)(-\xi \vartheta +\langle c,\bar x(1) \rangle) -\langle y^0 , A \bar x(1)  \rangle -  y_\tau + \tau \langle c,x\rangle  - \langle  \bar y(1) , z^0  \rangle - \tau \bar y_\tau(1) & \\
&=& -(\tau-1) \xi \vartheta -\langle  A^\top y^0-(\tau-1) c , \bar x(1) \rangle -  y_\tau - \langle  \bar y(1) ,\tau Ax+z^0 \rangle - \tau \bar y_\tau(1), & \text{$A^\top \bar y(1)=-c$,} \\
&=& -(\tau-1) \xi \vartheta -(\tau Ax+z^0, \tau, y, y_\tau)\bullet (A\bar x(1),1,\bar y(1), \bar y_\tau(1)).
\end{array}
\end{eqnarray}
For the final relation, by first using $y_\tau=y_{\tau,0}+\tau \langle c,x \rangle$ and then the third equality in \eqref{eq:dd-4-2} for $\mu$, and also $y_{\tau,0}= -\langle y^0, z^0 \rangle -\xi \vartheta$, we have
\begin{eqnarray} \label{eq:dd-outcome-36}
-(\tau Ax+z^0, \tau, y, y_\tau)\bullet (z^0,1,y^0, y_{\tau,0}) = \xi \vartheta (\mu+1). \nonumber 
\end{eqnarray}
By putting together the four inequality and equations we have for the three $\bullet$ products, we get
\begin{eqnarray} \label{eq:dd-outcome-37}
	\sigma_f  (\mu +1) \xi \vartheta  \leq (\tau-1) \xi \vartheta +\xi \vartheta  \left ( \frac{1}{\sigma_f} + \sigma _f \right).
\end{eqnarray}
By dividing both sides by $\xi \vartheta$ and reordering, we obtain the desired result. 
\end{proof}
This theorem and Lemma \ref{lem:dg-bound-1} imply a form of strong duality for the Domain-Driven form. 
\begin{corollary} (Strong Duality) Assume that both primal and dual problems are strictly feasible (there exists $x \in \R^n$ with $A x \in \inte D$ and $y \in \inte D_*$ such that $A^\top y=-c$). Then, there exist $x^* \in \mathcal F_p$ and $y^* \in \mathcal F_d$ such that $\langle c,x^* \rangle + \delta_*(y^*|D)=0$. 
\end{corollary}
\begin{proof}[Proof sketch]
For this proof, we can assume that $x^0 \in \mathcal F_p$ and $z^0=0$; therefore, for all the points on the central path we have $x(\mu) \in \mathcal F_p$. By Lemma \ref{lem:dg-bound-1}, the points on the central path $(x(\mu),\tau(\mu), y(\mu))$ satisfy the hypothesis of Theorem \ref{thm:fes-meas-dd} for every $\mu>0$ and so \eqref{eq:thm:fes-meas-dd-1} holds. This means $\lim_{\mu\rightarrow +\infty} \tau(\mu) = +\infty$. Also note that \eqref{eq:dg-bound-1} together with \eqref{eq:thm:fes-meas-dd-1} imply
\begin{eqnarray} \label{eq:SD-1}
\lim_{\mu \rightarrow +\infty} \left( \langle c, x(\mu) \rangle + \frac{1}{\tau(\mu)}\delta_*(y(\mu) | D)\right) =0. 
\end{eqnarray}
To complete the proof, we need to show there exists $\bar \mu >0$ such that the sets $\{x(\mu): \mu \geq \bar \mu\}$ and $\left\{{y(\mu)}/{\tau(\mu)}: \mu \geq \bar \mu\right\}$  are bounded and so have accumulation points $x^*$ and $y^*$ when $\mu \rightarrow +\infty$. For $x(\mu)$, $A\bar x(1) \in \inte D$ and $A^\top \Phi'(A \bar x(1)) = -c$ imply that there exists $\bar \mu$ such that for every $\mu \geq \bar \mu$, we have $
\langle \Phi'(A \bar x(1)), Ax(\mu)-A \bar x(1) \rangle =-\langle c, x(\mu)-\bar x(1) \rangle \geq 0
$. 
Then, by \cite{nemirovski-notes}-Lemma 3.2.1, the set $\{x(\mu): \mu \geq \bar \mu\}$ is bounded. For the dual part, if the set $\left\{{y(\mu)}/{\tau(\mu)}: \mu \geq \bar \mu\right\}$ is unbounded (seeking a contradiction), there exists a sequence $\{\mu^k\}$ such that $\|{y(\mu^k)}/{\tau(\mu^k)}\| \rightarrow +\infty$ and so $h:= \lim_{k\rightarrow+\infty} \frac{(y(\mu^k)/\tau(\mu^k)}{\|(y(\mu^k)/\tau(\mu^k)\|}$ is in $D_*$ and satisfies $A^\top h=0$. Then, $A\bar x(1) \in \inte D$ implies that $\frac{1}{\tau(\mu^k)}\delta_*\left(y(\mu^k)|D\right) \rightarrow +\infty$, which contradicts \eqref{eq:SD-1}. 
\end{proof}
Let us now consider a more general case where we just know that the primal and dual problems are feasible. 
\begin{lemma}  \label{lem:dd-11}
	Assume that there exist $ \bar x \in \mathcal F_p$ and $ \bar y \in \mathcal F_d$ with duality gap equal to $\langle c, \bar x \rangle + \delta_*(\bar y|D)=\Lambda$. For every point  $(x,\tau,y) \in Q_{DD}$,  $\kappa$-close to the central path, the variable $\tau$ satisfies 
	\begin{eqnarray} \label{eq:dd-outcome-3}
		((\xi-1) \vartheta-\kappa \sqrt{\vartheta}) \frac{\mu(x,\tau,y)}{\tau} \leq  \xi \vartheta + \langle y^0- \bar y, z^0-A \bar x  \rangle+\tau \Lambda.  
	\end{eqnarray}
\end{lemma}
\begin{proof}
Let $(x,\tau,y)$ be an arbitrary point in $Q_{DD}$. Utilizing $\langle y, A \bar x \rangle \leq \delta_*(y|D)$ in \eqref{eq:dg-bound-1}, we get
\begin{eqnarray} \label{eq:dd-outcome-4}
\frac{(\xi-1) \mu \vartheta-\mu\kappa \sqrt{\vartheta}}{\tau} \leq  -y_{\tau,0} - \tau \langle c , x \rangle -\langle y, A \bar x \rangle. 
\end{eqnarray}
Also note that from $\langle c, \bar x \rangle + \delta_*(\bar y|D)=\Lambda$ we have
\begin{eqnarray} \label{eq:dd-outcome-5}
\langle \bar y, z \rangle  \leq - \langle c, \bar x \rangle+\Lambda, \ \ \ \forall z \in D. 
\end{eqnarray}
Then,
\begin{eqnarray} \label{eq:dd-outcome-6}
-\langle c,  x \rangle &=&\langle  \bar y,  A x \rangle = \langle  \bar y,  A x + \frac{1}{\tau} z^0  \rangle -\langle  \bar y,  \frac{1}{\tau} z^0 \rangle  \nonumber \\
&\leq& - \langle c, \bar x \rangle+\Lambda -\langle  \bar y,  \frac{1}{\tau} z^0 \rangle,  \ \ \ \text{using \eqref{eq:dd-outcome-5}},
\end{eqnarray}
and also using $A^\top y = A^\top y^0 - (\tau-1) c$ we can easily get
\begin{eqnarray} \label{eq:dd-outcome-7}
\langle y, A \bar x \rangle = \langle  A^\top y^0 - (\tau-1) c ,  \bar x \rangle. 
\end{eqnarray}
Using \eqref{eq:dd-outcome-6}  and \eqref{eq:dd-outcome-7} in \eqref{eq:dd-outcome-4}, we have
\begin{eqnarray} \label{eq:dd-outcome-8}
\frac{(\xi-1) \mu \vartheta-\mu\kappa \sqrt{\vartheta}}{\tau} &\leq&  -y_{\tau,0}  - \langle c, \bar x \rangle -\langle  \bar y,  z^0 \rangle- \langle   y^0  ,  A \bar x \rangle +\tau \Lambda \nonumber \\
&=& \xi \vartheta+ \langle y^0 , z^0 \rangle - \langle c, \bar x \rangle -\langle  \bar y,  z^0 \rangle- \langle   y^0  ,  A \bar x \rangle + \tau \Lambda,
\end{eqnarray}
where the last equation is by substituting $y_{\tau,0}=-\langle y^0 , z^0 \rangle - \xi \vartheta$ from \eqref{starting-points-copy-2}. Then, we use $c=-A^\top \bar y$ to conclude \eqref{eq:dd-outcome-3}. 
\end{proof}
When there exist $\bar x \in \mathcal F_p$ and  $\bar y \in \mathcal F_d$ with zero duality gap,  we can rewrite \eqref{eq:dd-outcome-3} as
\begin{eqnarray} \label{eq:dd-outcome-3-2}
 \mu(x,\tau,y) \leq  \left[ \frac{  \xi \vartheta + \langle y^0- \bar y, z^0-A \bar x  \rangle}{(\xi-1) \vartheta-\kappa \sqrt{\vartheta}}\right] \tau, 
	\end{eqnarray}
which shows a lower bound for the rate of increase of $\tau$ with respect to $\mu$.

Similar to \cite{infea-2}, we define a point $(x,\tau,y) \in Q_{DD}$ an \emph{$\epsilon$-solution} of our problem if
\begin{eqnarray} \label{eq:eps-solution}
\max \left\{\frac{1}{\tau} \ , \ \frac{\vartheta \mu}{\tau^2} \right\} \leq \epsilon. 
\end{eqnarray}
Theorem \ref{thm:complexity-result}, Theorem \ref{thm:fes-meas-dd} and Lemma \ref{lem:dd-11} yield the following theorem for detecting an $\epsilon$-solution.

\begin{theorem}  \label{thm:solvable-com}
(a) Assume we have strict primal-dual feasibility for the Domain-Driven problem \eqref{main-p}. Then, the PtPCA algorithm returns an $\epsilon$-solution in number of iterations bounded by
\begin{eqnarray*}
	O\left(\sqrt{\vartheta} \ln \left(\frac{\vartheta}{\sigma_f \epsilon}\right)\right).
\end{eqnarray*}

\noindent (b) Assume there exist $\bar x \in \mathcal F_p$ and  $\bar y \in \mathcal F_d$ with zero duality gap. In view of Lemma \ref{lem:dd-11}, let 
\[
B:=\frac{\xi \vartheta + \min \{ \langle y^0- \bar y, z^0-A \bar x  \rangle: \langle c,\bar x\rangle+\delta_*(\bar y|D) =0, \bar x \in \mathcal F_p, \bar y \in \mathcal F_d \}}{(\xi-1)\vartheta-\kappa \sqrt{\vartheta}}.
\]
Then, the PtPCA algorithm returns an $\epsilon$-solution in number of iterations bounded by
\begin{eqnarray*}
	O\left(\sqrt{\vartheta} \ln \left(\frac{\vartheta B}{ \epsilon}\right)\right).
\end{eqnarray*}
\end{theorem}

\section{Weak infeasibility and unboundedness detector} \label{sec:weak}
We start this section by showing that for the points close to the central path, variable $\tau$ stays away from zero. 
\begin{lemma} \label{lem:tau_bound}
Consider two points $(x,\tau,y)$ and $(\bar x,\bar \tau, \bar y)$ in $Q_{DD}$ with $\mu:=\mu(x,\tau,y)$ and $\bar \mu:=\mu(\bar x,\bar \tau, \bar y)$ that are $\kappa$-close to the central path. Then
\begin{eqnarray} \label{eq:bound_tau_3}
\mu \bar \tau^2+\bar \mu \tau^2 \leq  \frac{\xi}{\xi-1-\kappa} \tau \bar \tau(\mu + \bar \mu). 
\end{eqnarray}
\end{lemma}
\begin{proof}
By considering the fact that $\langle \frac{y}{\tau}, A \bar x + \frac{1}{\bar \tau} z^0 \rangle \leq \frac{1}{\tau} \delta_*(y|D) $, multiplying both sides of the RHS inequality in \eqref{eq:dg-bound-1} with $\tau \bar \tau$, using $A^\top y = A^\top y^0-(\tau-1)c$, and reordering the terms we get
\begin{eqnarray} \label{eq:bound_tau_1}
\langle y,z^0 \rangle + \tau \bar \tau\langle c,x \rangle -\tau \bar \tau\langle c, \bar x \rangle +\langle A^\top y^0+c,  \bar \tau \bar x  \rangle  +\bar \tau y_{\tau,0} \leq  - \frac{(\xi-1-\kappa)\mu \bar \tau  \vartheta}{\tau}.
\end{eqnarray}
If we use the third line of \eqref{eq:dd-4-2} for $\bar \mu$, we can simplify \eqref{eq:bound_tau_1} as
\begin{eqnarray} \label{eq:bound_tau_2}
\langle y,z^0 \rangle -\langle \bar y,z^0 \rangle+ \tau \bar \tau\langle c,x \rangle -\tau \bar \tau\langle c, \bar x \rangle  \leq  - \frac{(\xi-1-\kappa)\mu \bar \tau  \vartheta}{\tau} + \xi \bar \mu \vartheta.
\end{eqnarray}
Considering the symmetry we have in \eqref{eq:bound_tau_2}, if we change the role of $(x,\tau,y)$ and $(\bar x,\bar \tau, \bar y)$ and repeat the argument, we get a similar inequality as \eqref{eq:bound_tau_2} with the LHS negated and $\mu$ and $\bar \mu$ swapped in the RHS. By adding these two inequalities and canceling out $\vartheta$ from both sides, we get 
\begin{eqnarray} \label{eq:bound_tau_4}
0 \leq -(\xi-1-\kappa)\left(\frac{\mu \bar \tau }{\tau} + \frac{\bar \mu \tau }{\bar \tau}\right)+ \xi (\mu + \bar \mu). 
\end{eqnarray}
By reordering the terms we obtain \eqref{eq:bound_tau_3}. 
\end{proof}
\begin{corollary}\label{cor:bound_tau}
For all the points $(x,\tau,y) \in Q_{DD}$ with $\mu(x,\tau,y) \geq 1$ which are $\kappa$-close to the central path, we have 
\begin{eqnarray} \label{eq:tau_bound}
\tau \geq \tau_{\xi,\kappa}:= \frac{\xi-1-\kappa}{2\xi}.
\end{eqnarray}
\end{corollary}
Assume that the problem is (strongly) infeasible or unbounded, but not ill-posed. Then, at least one of $\sigma_p$ or $\sigma_d$ defined in \eqref{tf} is positive and Lemma \ref{lem:dd-con-par} implies that $\tau$ is bounded. Because $\tau$ is bounded and we have 
\[
\frac{\tau}{\mu} A^\top y = \frac{ \tau}{\mu} A^\top y^0 - \frac{\tau (\tau-1)}{\mu} c,
\]
$\frac{\tau}{\mu}y$ converges to a point in the kernel of $A^\top$ when $\mu$ goes to $+\infty$. If we can confirm that $\frac{\tau}{\mu} \delta_* (y|D) < 0$, then we have an approximate certificate of infeasibility. On the other hand, if $\langle c,x \rangle$ becomes a very large negative number, then $Ax$ dominates the other term in $Ax+\frac{1}{\tau} z^0$ (by Corollary \ref{cor:bound_tau}, $\tau \geq \tau_{\xi,\kappa}$) and we have an approximate certificate of unboundedness; since for every vector $y_c$ such that $A^\top y_c =-c$, we have $\|y_c\| \|Ax\| \geq |\langle c, x \rangle|$.

\noindent We say $(x,\tau,y) \in Q_{DD}$ is an \emph{$\epsilon$-certificate of infeasibility} if it satisfies  
\begin{eqnarray} \label{eq:epsilon-cert-2}
	\frac{\tau}{\mu}  \delta_* \left (\left.y\right |D \right)  < -1, \ \ \ \frac{\tau}{\mu} \|A^\top y\| \leq \epsilon. 
\end{eqnarray}
We say $(x,\tau,y) \in Q_{DD}$ is an \emph{$\epsilon$-certificate of unboundedness} if it satisfies  
\begin{eqnarray} \label{eq:epsilon-cert-3}
	\langle c,x \rangle   < -\frac{1}{\epsilon}.
\end{eqnarray}
When we are $\kappa$-close to the central path, by Lemma \ref{lem:dg-bound-1}, we have
\begin{eqnarray} \label{eq:epsilon-cert-4}
\langle c,x \rangle +\frac{1}{\tau} \delta_*\left(\left.y\right|D \right)  \leq \frac{-y_{\tau,0}}{\tau} - \left ((\xi-1)-\frac{\kappa}{\sqrt{\vartheta}}\right)\frac{\mu \vartheta }{\tau^2}.
\end{eqnarray}
Using this inequality, we can prove the following theorem:
\begin{theorem} \label{thm:weak-detector} (weak detector) Assume that at least one of $\sigma_p$ or $\sigma_d$ is positive. Then, the PtPCA  algorithm returns either an $\epsilon$-certificate of infeasibility or an $\epsilon$-certificate of unboundedness in number of iterations bounded by
\begin{eqnarray} \label{eq:epsilon-cert-5}
O\left(\sqrt{\vartheta} \ln \left( \frac{1}{\vartheta\epsilon}  \min \left\{ \frac{\|z^0\|}{\sigma_p},\frac{\|y^0\|}{\sigma_d} \right\} \right)\right).
\end{eqnarray}
\end{theorem}
\begin{proof}
In view of \eqref{eq:epsilon-cert-4}, we want $\frac{\mu \vartheta }{\tau^2}$ to be as large as $O(1/\epsilon)$ and $\tau$ is bounded by the result of Lemma \ref{lem:dd-con-par}. We also know how the PtPCA algorithm  increases $\mu$ by Theorem \ref{thm:complexity-result}. We can assume that  $ (\xi-1)-\frac{\kappa}{\sqrt{\vartheta}} > 0$, then, when $\frac{\mu \vartheta }{\tau^2}$ gets large enough, \eqref{eq:epsilon-cert-4} implies that at least one of \eqref{eq:epsilon-cert-2} or \eqref{eq:epsilon-cert-3} happens. Putting together these facts  gives us the desired result. 
\end{proof}
Let us see how the weak detector behaves in the infeasibility and unbounded cases we defined above. 
\subsection{Infeasibility} 
If the problem is infeasible, but not ill-posed, we have $\sigma_p >0$, and so by Lemma \ref{lem:dd-con-par}, $t_p(z^0) < +\infty$. 
For a given $t >0$, let us define
\begin{eqnarray} \label{eq:bound-1}
B_{p,t}:=\sup \left\{ \|x\|: \exists \tau \in (t,t_p(z^0)) \ \ \text{s.t.} \ Ax+\frac{1}{\tau} z^0 \in D \right\}. 
\end{eqnarray}
By using \eqref{eq:CS}, for every point $(x,\tau,y) \in Q_{DD}$, we have
\begin{eqnarray} \label{eq:dd-outcome-9}
-\langle c ,x \rangle \leq \|c\| B_{p,\tau_{\xi,\kappa}}. 
\end{eqnarray}
We will show that strict  infeasibility is a sufficient condition for $B_{p,t} < +\infty$ for every $t \in (0,t_p(z^0))$. If we multiply both sides of \eqref{eq:dg-bound-1} by $\tau^2/\mu$ and reorder the terms, we have
\begin{eqnarray} \label{eq:lem:dd-13-1}
\frac{\tau}{\mu} \delta_* \left (y|D \right)  \leq \frac{\tau}{\mu}[-y_{\tau,0} - \tau \langle c , x \rangle]  - [(\xi-1)\vartheta-\kappa \sqrt{\vartheta}].
\end{eqnarray}
Therefore, when $-\langle c ,x \rangle$ is bounded, for every point $(x,\tau,y) \in Q_{DD}$ $\kappa$-close to the central path with a large  $\mu$, we have $\delta_*(y|D) < 0$. By the proof of Theorem \ref{thm:weak-detector} and \eqref{eq:dd-outcome-9}, the weak detector returns an $\epsilon$-certificate of infeasibility in number of iterations bounded by
\[
O \left ( \sqrt{\vartheta}  \ln \left (t_p(z^0)  B_{p,\tau_{\xi,\kappa}} +  \frac{t_p(z^0)}{\vartheta} \frac{1}{\epsilon}  \right )\right).
\]


\subsection{{Unboundedness}}
If the problem is unbounded, but not ill-posed, we have $\sigma_d >0$, and so $t_d(y^0) < +\infty$ for every $y^0 \in \inte D_*$ by Lemma \ref{lem:dd-con-par}. For a given $t >0$, let us define
\begin{eqnarray} \label{eq:bound-2}
B_{d,t}:=\sup \left\{ \|y\|: y \in D_*, \exists \tau \in (t,t_d(y^0)) \ \ \text{s.t.} \ A^\top y= A^\top y^0-(\tau-1)c \right\}. 
\end{eqnarray}
We will show that strict unboundedness is a sufficient condition for $B_{d,t} < +\infty$ for every $t \in (0,t_d(y^0))$. Then, for every $(x,\tau,y) \in Q_{DD}$ we have
\begin{eqnarray} \label{eq:dd-outcome-39-2}
\frac{\tau}{\mu} \delta_* \left(\left. y\right| D\right)\geq \frac{\tau}{\mu} \langle y, z^0 \rangle  \geq - \frac{B_{d,\tau_{\xi,\kappa}}}{\mu} t_d(y^0) \|z^0\|.
\end{eqnarray}
Hence, for every point $(x,\tau,y) \in Q_{DD}$ with $\mu \geq 2B_{d,\tau_{\xi,\kappa}} t_d(y^0) \|z^0\|$  we have $\frac{\tau}{\mu}\delta_* \left( \left.y\right|D\right)  \geq -\frac 12$. Therefore, by Theorem \ref{thm:weak-detector}, our weak detector returns an $\epsilon$-certificate of unboundedness in polynomial time. 
In fact, for every $\epsilon > 0$, by using the proof of Theorem \ref{thm:weak-detector}, after at most
\begin{eqnarray} \label{eq:dd-outcome-25-2}
	O \left ( \sqrt{\vartheta}  \ln \left (t_d(y^0)  B_{d,\tau_{\xi,\kappa}} +  \frac{t_d(y^0)}{\vartheta} \frac{1}{\epsilon}  \right )\right)
\end{eqnarray}
iterations, the weak detector returns an $\epsilon$-certificate of unboundedness.  
\section{Strict infeasibility and unboundedness detector} \label{sec:strict}
In this section, we show that in the case of strict infeasibility and unboundedness, we can actually find a certificate for the exact problem in polynomial time. The idea is that we need to project our current point onto a proper linear subspace using a suitable norm.
\subsection{Infeasibility}
By definition, if there exists  $ y \in  \inte{D_*}$ such that $A^\top  y=0$ and $\delta_*(y|D)=-1$, we have strict primal infeasibility. To get an exact certificate, we show how to project our current point $y$ onto the kernel of $A^\top$ with respect to a suitable norm. 
Let us first show that for all the points $(x,\tau)$ such that $Ax+\frac{1}{\tau}z^0 \in D$,  norm of $x$ is bounded. 
\begin{lemma}  \label{lem:dd-16}
Assume that there is a point $y \in \inte D_*$ such that $A^\top y =0$ and $\delta_*(y|D)=-1$. Then, $t_p(z^0) < +\infty$ and for a given $t \in (0,t_p(z^0))$, we have $B_{p,t} < +\infty$. 
\end{lemma}
\begin{proof}
By Lemma \ref{lem:outcome-infeas} we have $\sigma_p>0$ and so by Lemma \ref{lem:dd-con-par}, $t_p(z^0) < +\infty$. Suppose that $B_{p,t} = +\infty$ (we are seeking a contradiction). Then, since $\tau$ is bounded by $t_p(z^0)$, the set $D$ must have a nonzero recession direction in the range of $A$; we write it as $Ah$. Consider a point $z \in \inte D$ such that $A^\top \Phi'(z)=0$, which exists because we have a point $y \in \inte D_*$ such that $A^\top y =0$. Then, by a property of s.c.\ barriers (see for example \cite{nemirovski-notes}-Corollary 3.2.1), we have
\begin{eqnarray} \label{eq:lem:dd-16-3}
0= \langle \Phi'(z), Ah \rangle  \geq \sqrt{\langle Ah, \Phi''(z) Ah \rangle}  \ \ \Rightarrow \ \ Ah=0.
\end{eqnarray}
This is a contradiction. 
\end{proof}
For the main proof of this section, we define a set of points that get close to the points on the central path for large enough $\mu$. Consider the following optimization problem for each $\tau \geq \tau_{\xi,\kappa}$, where $\tau_{\xi,\kappa}$ is the lower bound we have for $\tau$ by Corollary \ref{cor:bound_tau}. 
\begin{eqnarray} \label{eq:dd-outcome-14}
	\begin{array}{crcl}
		\min & \Phi_*(y) &&\\
		& A^\top y &=&0 \\
		& \langle y, z^0 \rangle &=&-\tau \xi \vartheta.
	\end{array}
\end{eqnarray}
Note that this problem is feasible by strict infeasibility assumption. It also has an optimal solution for every $\tau$ such that there exists $x$ with $Ax +\frac{1}{\tau}z^0 \in D$. This holds since the s.c.\ function $\Phi_*$ is bounded from below on the feasible region \cite{interior-book}-Theorem 2.2.3; by Fenchel-Young inequality:
\[
\Phi_*(y) \geq \langle y,Ax+\frac{1}{\tau} z^0 \rangle - \Phi\left(Ax+\frac{1}{\tau} z^0\right) = -\xi\vartheta - \Phi\left(Ax+\frac{1}{\tau} z^0\right).
\] Let us denote the solution of this problem by $\bar y(\tau)$. If we write the optimality conditions for $\bar y(\tau)$, we have $\Phi'_*(\bar y(\tau)) = A \bar x(\tau) + \frac{1}{t(\tau)} z^0$, for some $\bar x(\tau)$ and $t(\tau)$. We claim that $t(\tau) \geq \frac{\tau_{\xi,\kappa}\xi}{\tau_{\xi,\kappa}\xi+1}$.  By \cite{interior-book}-Theorem 2.4.2, we have
\begin{eqnarray} \label{eq:dd-outcome-38}
	\langle \bar y(\tau),z^0 \rangle -\vartheta  \leq \langle \bar y(\tau), \Phi'_*(\bar y(\tau)) \rangle= \langle \bar y(\tau), A \bar x(\tau) + \frac{1}{t(\tau)} z^0 \rangle= \frac{1}{t(\tau)} \langle \bar y(\tau),z^0 \rangle  \nonumber \\
	\Rightarrow  \ \ -\tau \xi \vartheta - \vartheta  \leq \frac{-1}{t(\tau)}  \tau  \xi \vartheta \ \ \Rightarrow  \ \ \frac{1}{t(\tau)}    \leq \frac{\tau \xi +1}{\tau \xi}  \leq \frac{ \tau_{\xi,\kappa}\xi +1}{ \tau_{\xi,\kappa}\xi}.
\end{eqnarray}
Now we are ready to prove the following lemma which shows $ \bar y(\tau)$ gets very close to $\frac{\tau}{\mu} y$ in the local norm when $\mu$ is large enough. For the proof, we use a property that if $f$ is a s.c.\ function, then, for every  $x$ and $y$ in its domain, we have \cite{lectures-book}-Theorem 5.1.8  ($r:=  \|y-x\|_{f''(x)}$)
	\begin{eqnarray} \label{property-ineq-1-2}
	\langle f'(x) - f'(y) , y-x \rangle \geq  \frac{r^2}{1+r}.
	\end{eqnarray}
\begin{lemma}  \label{lem:dd-17}
For $\mu \geq 1$, consider a point $(x,\tau,y):=(x(\mu),\tau(\mu),y(\mu))$ on the central path and $\bar y(\tau)$ as the solution of \eqref{eq:dd-outcome-14}. Then, there exists a constant $B_{inf}$ depending on the initial point $z^0$, $B_{p,\tau_{\xi,\kappa}}$, and $t_p(z^0)$ such that 
\begin{eqnarray} \label{eq:dd-outcome-20}
\frac{\left\|\bar y(\tau) - \frac{\tau}{\mu} y\right\|^2_{\Phi_*''(\frac{\tau}{\mu} y)}}{1+\left\|\bar y(\tau) - \frac{\tau}{\mu} y\right\|_{\Phi_*''(\frac{\tau}{\mu} y)}}   \leq   \frac{B_{inf}}{\mu}.
\end{eqnarray}
\end{lemma}
\begin{proof}
Note that by the definition of $\tau_{\xi,\kappa}$ and our choices of $\kappa$ and $\xi$, we have $\tau_{\xi,\kappa}\leq  \frac{\tau_{\xi,\kappa} \xi}{\tau_{\xi,\kappa} \xi+1}$. 
By using property \eqref{property-ineq-1-2} of s.c.\ functions, we have
\begin{eqnarray} \label{eq:dd-outcome-15}
\frac{\left\|\bar y(\tau) - \frac{\tau}{\mu} y\right\|^2_{\Phi_*''(\frac{\tau}{\mu} y)}}{1+\left\|\bar y(\tau) - \frac{\tau}{\mu} y\right\|_{\Phi_*''(\frac{\tau}{\mu} y)}}   \leq   \langle \frac{\tau}{\mu} y - \bar y(\tau) , Ax + \frac{1}{\tau} z^0 -A \bar x(\tau) -  \frac{1}{t(\tau)} z^0 \rangle.
\end{eqnarray}
Because $\bar y(\tau)$ is the solution of \eqref{eq:dd-outcome-14}, we have
\begin{eqnarray} \label{eq:dd-outcome-16}
		\langle  -\bar y(\tau) , Ax + \frac{1}{\tau} z^0 -A \bar x(\tau) -  \frac{1}{t(\tau)} z^0 \rangle = \left( \frac{1}{\tau} -  \frac{1}{t(\tau)}\right) \tau \xi \vartheta. 
\end{eqnarray}
We also have
\begin{eqnarray} \label{eq:dd-outcome-17}
		\langle  \frac{\tau}{\mu} y , Ax + \frac{1}{\tau} z^0 -A \bar x(\tau) -  \frac{1}{t(\tau)} z^0 \rangle =  \frac{\tau}{\mu} \langle y , A x - A \bar x (\tau) \rangle +  \left( \frac{1}{\tau} -  \frac{1}{t(\tau)}\right) \langle \frac{\tau}{\mu} y, z^0\rangle. 
\end{eqnarray}
For the first term of \eqref{eq:dd-outcome-17}, using Lemma \ref{lem:dd-16}, we have
\begin{eqnarray} \label{eq:dd-outcome-18}
\frac{\tau}{\mu} \langle y , A x - A \bar x (\tau) \rangle &=&  \frac{\tau}{\mu} \langle A^\top y ,  x -  \bar x (\tau) \rangle  =  \frac{\tau}{\mu} \langle A^\top y^0+(\tau-1)c ,  x -  \bar x (\tau) \rangle\nonumber \\
&\leq& \frac{t_p(z^0)}{\mu} \left(2B_{p,\tau_{\xi,\kappa}}(\|A^\top y^0\|+(t_p(z^0)-1)\|c\|)\right). 
\end{eqnarray}
For the second term of \eqref{eq:dd-outcome-17}, by using the third line in \eqref{eq:dd-4-2}, we have
\begin{eqnarray} \label{eq:dd-outcome-19}
\frac{\tau}{\mu}  \langle y , z^0 \rangle & =&  \frac{ \tau^2 (-y_{\tau,0} -  \langle c,  x \rangle - \langle y^0 , A x  \rangle)}{\mu} - \tau \xi \vartheta   \nonumber \\
&\leq & \frac{t^2_p(z^0)}{\mu} \left(|y_{\tau,0}|+B_{p, \tau_{\xi,\kappa}}\|A^\top y^0+c\|\right)- \tau \xi \vartheta. 
\end{eqnarray}
\eqref{eq:dd-outcome-18} and \eqref{eq:dd-outcome-19} give an upper bound for \eqref{eq:dd-outcome-17}, which we add to \eqref{eq:dd-outcome-16} to get an upper bound for the RHS of \eqref{eq:dd-outcome-15}.     Therefore, \eqref{eq:dd-outcome-20} holds for 
\begin{eqnarray*} \label{eq:dd-outcome-21}
B_{inf}:=t_p(z^0) \left(2B_{p,\tau_{\xi,\kappa}}(\|A^\top y^0\|+(t_p(z^0)-1)\|c\|)\right)+\frac{2t^2_p(z^0)}{\tau_{\xi,\kappa}} \left(|y_{\tau,0}|+B_{p,\tau_{\xi,\kappa}} \|A^\top y^0+c\| \right). 
\end{eqnarray*}
\end{proof}
Now we can prove the following proposition for our strict detector:
\begin{proposition} \label{prop:strict-infea} (Strict primal infeasibility detector) Assume that there exists $\hat y \in \inte D_*$ such that $A^\top \hat y = 0$ and $\delta_*(\hat y|D) =-1$. Then, a modification of the PtPCA algorithm returns a point $y \in D_*$ with $A^\top y =0$ and $\delta_*(y|D) \leq -1$ in number of iterations bounded by
\begin{eqnarray} \label{eq:prop:strict-infea}
O\left ( \sqrt{\vartheta}  \ln \left(t_p(z^0)B_{p,\tau_{\xi,\kappa}} \right) \right). 
\end{eqnarray}
\end{proposition}
\begin{proof}
Assume that $\kappa$ is chosen such that for every point $(x,\tau,y)$ $\kappa$-close to the central path, we have $\|\frac{\tau(\mu)}{\mu} y(\mu) - \frac{\tau}{\mu} y\|_{\Phi_*''(\frac{\tau}{\mu} y)}\leq 0.1 $, where $\mu:=\mu(x,\tau,y)$, and $|\frac{\tau(\mu)}{\tau}-1| \leq 0.1$ \cite{karimi_arxiv}.  If the PtPCA algorithm is run until  we have $\mu \geq 10 B_{inf}$,  we get $\|\bar y(\tau(\mu)) - \frac{\tau(\mu)}{\mu} y(\mu)\|_{\Phi_*''(\frac{\tau(\mu)}{\mu} y(\mu))} \leq 0.1$ in view of Lemma \ref{lem:dd-17}. Now, if we project $\frac{\tau}{\mu} y$ on the set $\{y: A^\top y=0, \langle y, z^0 \rangle \leq - 0.9 \tau \xi \vartheta\}$ by the norm defined by $\Phi_*''(\frac{\tau}{\mu} y)$, the projection $\hat y$ must have a distance (in local norm) to $\bar y(\tau(\mu))$ smaller than $1$ and so it lies in $\inte D_*$. We just need to show that $\delta_*(\hat y|D) < 0$. We have (with $u:= Ax + \frac{1}{\tau} z^0$)
	\begin{eqnarray} \label{eq:dd-outcome-22}
		\langle \hat y, \Phi'_*(\hat y) \rangle &=& \langle \hat y, \Phi'_*(\hat y)-   u\rangle + \langle \hat y, Ax + \frac{1}{\tau} z^0 \rangle \nonumber \\
		&\leq &\|\hat y\|_{\Phi''_*(\hat y)} \|\Phi'_*(\hat y)-  u \|_{[\Phi''_*(\hat y)]^{-1}} + \frac{1}{\tau} \langle \hat y, z^0 \rangle,   \ \ \  \text{using $A^\top \hat y =0$} \nonumber \\
		&\leq&  \sqrt{\vartheta} \|\Phi'_*(\hat y)-  u \|_{[\Phi''_*(\hat y)]^{-1}} - 0.9 \xi \vartheta,  \ \ \  \text{using $\|\hat y\|_{\Phi''_*(\hat y)}  \leq \sqrt{\vartheta}$}.
	\end{eqnarray}
Since $(x,\tau,y)$ is $\kappa$-close to the central path, $\|\Phi'_*(\hat y)-  u\|_{[\Phi''_*(\hat y)]^{-1}}$ is smaller than 1, and so if $\xi$ is chosen properly we have $\langle \hat y, \Phi'_*(\hat y) \rangle  \leq -\tilde \xi \vartheta$ for some $\tilde \xi > 1$. By \cite{interior-book}-Theorem 2.4.2 we have $\delta_*(\hat y|D) <0$ as we want.
\end{proof}
The following remark shows that in the special case of conic optimization, the complexity result of Proposition \ref{prop:strict-infea} reduces to the one in  \cite{infea-2}.  
\begin{remark}
Assume that $D=K-b$ where $K$ is a convex cone equipped with a $\vartheta$-LH.s.c.\ barrier $\hat \Phi$ and that the system $Ax+b=z$, $z \in K$, is strictly infeasible. Then, for a given $z^0 \in \inte K$, the infeasibility measure in \cite{infea-2} is defined as
\begin{eqnarray} \label{eq:trict-infea-comp-1}
\rho_d:=\max\{\alpha: z^0-\alpha {\hat z} \in K, 1-\alpha {\hat\tau} \geq 0\},
\end{eqnarray}
where $({\hat z}, {\hat \tau})$ is the optimal solution of 
\begin{eqnarray} \label{eq:trict-infea-comp-2}
\min \{\hat \Phi( z)-\ln( \tau) : Ax+ z-  \tau b= z^0-b, x \in \R^n,  z \in K,  \tau \geq 0\}. 
\end{eqnarray}
It is proved in \cite{infea-2} that $\rho_d \leq 1$. By \cite{infea-2}-Lemma 4, for every feasible solution $(z,\tau)$ of problem \eqref{eq:trict-infea-comp-2}, we have
\begin{eqnarray} \label{eq:trict-infea-comp-3}
\|z-z^0\|^*_{\hat\Phi''(z^0)}+(\tau-1)^2 \leq \frac{\vartheta (\vartheta+1)}{\rho_d}. 
\end{eqnarray}
We claim that this inequality gives a bound for both $t_p(z^0)$ and $B_{p,\tau_{\xi,\kappa}}$ in \eqref{eq:prop:strict-infea}. As $D=K-b$, having a point $x \in \R^n$ and $\tau >0$ such that $Ax+\frac{1}{\tau} z^0 \in D$ implies there exists $z \in K$ such that 
\begin{eqnarray} \label{eq:trict-infea-comp-4}
Ax+\frac{1}{\tau} z^0=z-b \ \  \Rightarrow  \  A(-\tau x)+\tau z -(\tau+1)b=z^0-b. 
\end{eqnarray}
Therefore, by using \eqref{eq:trict-infea-comp-3}, both $t_p(z^0)$ and $B_{p,\tau_{\xi,\kappa}}$ are $O(\frac{\vartheta^2}{\rho_d})$ and the complexity bound in \eqref{eq:prop:strict-infea} reduces to 
\begin{eqnarray} \label{eq:trict-infea-comp-5}
O\left ( \sqrt{\vartheta}  \ln \left(\frac{\vartheta}{\rho_d}\right) \right),
\end{eqnarray}
as in \cite{infea-2}. 
\end{remark}

\subsection{Unboundedness}
We start with a lemma that if the problem is  strictly unbounded, the set of $y$ vectors is bounded. 
\begin{lemma} \label{lem:bound-2}
Assume that problem \eqref{main-p} is strictly unbounded. Then, for each given $t \in (0,t_d(y^0))$, we have $B_{d,t} < +\infty$. 
\end{lemma}
\begin{proof}
In view of the definition of strict unboundedness, take $h \in \R^n$ such that $Ah \in \inte(\rec(D))$.  By definition of $D_*$ in \eqref{eq:leg-conj-2}, for every $y \in D_*$ we have $y^\top Ah <0$. Hence,  the intersection of $D_*$ with the kernel of $A^\top$ is the zero vector.  Therefore, the set of $y$ vectors in $ D_*$ which satisfy $A^\top y=A^\top y^0 -(\tau-1)  c$ for a $\tau \in (t, t_d(y^0))$ is bounded.
\end{proof}
For the case of strict unboundedness, consider a point $(x(\mu), \tau(\mu), y(\mu))$ on the central path for parameter $\mu$.  For a fixed $\mu>0$, let  $\bar x(\mu)$ be the unique solution of the following problem:
\begin{eqnarray} \label{eq:dd-outcome-26}
	\begin{array}{cr}
		\min & \frac{\mu}{\tau(\mu)}\Phi(Ax)-\langle A^\top y^0+c, x \rangle\\
		 \textup{s.t.} & \langle c,x \rangle \leq \langle c, x(\mu) \rangle,
	\end{array}
\end{eqnarray}
and  $\bar y(\mu) := \frac{\mu}{\tau(\mu)}\Phi'(A\bar x(\mu))$. To show \eqref{eq:dd-outcome-26} has an optimal solution, we must prove that the s.c.\ function $\frac{\mu}{\tau(\mu)}\Phi(Ax)-\langle A^\top y^0+c, x \rangle$ is bounded from below \cite{interior-book}-Theorem 2.2.3. The strict unboundedness implies that there exists $\hat y \in \inte D_*$ such that $A^\top y^0+c=A^\top \hat y$. Then, by Fenchel-Young inequality:
\[
\frac{\mu}{\tau(\mu)}\Phi(Ax)-\langle A^\top y^0+c, x \rangle  \geq -\frac{\mu}{\tau(\mu)}\Phi_*\left(\frac{\tau(\mu)}{\mu} \hat y \right). 
\]
By the first order optimality condition for \eqref{eq:dd-outcome-26}, we have 
\[
A^\top \bar y(\mu) = A^\top y^0+c - \bar t(\mu)c,
\]
for some $\bar t(\mu) \geq 0$. Note that $\bar t(\mu)=0$ if  $\langle c,\bar x(\mu) \rangle < \langle c, x(\mu) \rangle$.  Then, by using property \eqref{property-ineq-1-2}, we have (with $u(\mu):=Ax(\mu)+\frac{1}{\tau(\mu)} z^0$)
\begin{eqnarray} \label{eq:dd-outcome-27}
	\frac{\|u(\mu) - A \bar x (\mu)\|^2_{\Phi''(u(\mu))}}{1+\|u(\mu) - A \bar x (\mu)\|_{\Phi''(u(\mu))}} &\leq& \frac{\tau(\mu)}{\mu}\langle y(\mu)-\bar y(\mu), Ax(\mu)+\frac{1}{\tau(\mu)} z^0 - A \bar x (\mu) \rangle  \nonumber \\
	&=& -\frac{\tau(\mu)(\tau(\mu)-\bar t(\mu))}{\mu} \langle c , x(\mu)-\bar x(\mu) \rangle + \frac{1}{\mu}  \langle y(\mu)-\bar y(\mu), z^0 \rangle  \nonumber  \\
	&\leq&  \frac{1}{\mu}  \langle y(\mu)-\bar y(\mu), z^0 \rangle.
\end{eqnarray}
To explain the last inequality, if  $\langle c , x(\mu)-\bar x(\mu) \rangle=0$, we have equality, otherwise, $\bar t(\mu)=0$ and inequality holds as $\bar x(\mu)$ is feasible for \eqref{eq:dd-outcome-26}. 
Using Lemma \ref{lem:bound-2}, we get our strict unboundedness detector.

\begin{proposition} \label{prop:strict-unb} (Strict primal unboundedness detector) Assume that the primal problem \eqref{main-p} is strictly unbounded. Then, for every $\epsilon>0$, after running PtPCA algorithm at most 
\begin{eqnarray} \label{eq:prop:strict-unb}
O \left ( \sqrt{\vartheta}  \ln \left (\frac{1}{\vartheta \epsilon} t_d(y^0)+ t_d(y^0) B_{d,0} \right )\right)
\end{eqnarray}
iterations, the projection of $u:= Ax+\frac{1}{\tau} z^0$ for the current point $(x,\tau,y)$ with respect to the norm defined by $\Phi''(u)$ into the set $\{z: z=Ax, \  \langle c, x \rangle \leq -1/\epsilon\}$ yields a point $\bar x \in \R^n$ such that  $A \bar x \in \inte(D)$ and $\langle c, \bar x \rangle \leq -\frac{1}{\epsilon}$. 
\end{proposition}
\begin{proof}
Using inequality \eqref{eq:dd-outcome-27} and the definition of $B_{d,t}$ in \eqref{eq:bound-2}, we have
\begin{eqnarray} \label{eq:dd-outcome-27-2}
\frac{\|u(\mu) - A \bar x (\mu)\|^2_{\Phi''(u)}}{1+\|u(\mu) - A \bar x (\mu)\|_{\Phi''(u)}} \leq  \frac{2}{\mu}  B_{d,0} \|z^0\|.
\end{eqnarray}
Thus, for every scalar $\delta \in (0,1)$, if $\mu \geq \frac{2}{\delta} B_{d,0} \|z^0\|$, then $\|u - A \bar x (\mu)\|_{\Phi''(u)}\leq \delta/(1-\delta)$.  On the other hand, when our current point $(x,\tau,y)$ is $\kappa$-close to the central path, $\|u-u(\mu)\|_{\Phi''(u)}$ is sufficiently smaller than 1, where $u:= Ax +\frac 1\tau z^0$. Therefore, when $\mu$ is large enough, the projection of $u$ with respect to the norm defined by $\Phi''(u)$ into the set $\{z: z=Ax, \  \langle c, x \rangle \leq -\frac{1}{\epsilon}\}$ is in $\inte D$, using the Dikin ellipsoid property of s.c.\ functions.  Also note that after at most the number of iterations given in \eqref{eq:dd-outcome-25-2}, we get a point $x$ with $\langle c,x \rangle \leq -\frac{1}{\epsilon}$. Putting these two together, we get the statement of the proposition.  
\end{proof}
\begin{remark}
Let $D=K-b$ for $K$ a convex cone. Then, $\rec(D)=K$ and we have $D_*=K_*:=-K^*$, where $K^*$ is the dual cone of $K$ we defined before. 
Assume that $K_*$ is equipped with a $\vartheta$-LH.s.c.\ barrier  $\hat \Phi_*$ and that the system $A^\top y=-c$, $y \in K_*$, is strictly infeasible. Then,  for a given $y^0 \in \inte K_*$, the infeasibility measure in \cite{infea-2} is defined as
\begin{eqnarray} \label{eq:trict-infea-comp-1-u}
\rho_p:=\max\{\alpha: y^0-\alpha {\hat y} \in K_*, 1-\alpha {\hat\tau} \geq 0\},
\end{eqnarray}
where $({\hat y}, {\hat \tau})$ is the optimal solution of 
\begin{eqnarray} \label{eq:trict-infea-comp-2-u}
\min \{\hat \Phi_*(y)-\ln( \tau) : A^\top y+ \tau c= A^ \top y^0+c, y \in K_*,  \tau \geq 0\}. 
\end{eqnarray}
It is proved in \cite{infea-2} that $\rho_p \leq 1$. By \cite{infea-2}-Lemma 4, for every feasible solution $(y,\tau)$ of problem \eqref{eq:trict-infea-comp-2-u}, we have
\begin{eqnarray} \label{eq:trict-infea-comp-3-u}
\|y-y^0\|^*_{\hat\Phi_*''(y^0)}+(\tau-1)^2 \leq \frac{\vartheta (\vartheta+1)}{\rho_p}. 
\end{eqnarray}
This inequality shows that $t_d(y^0)$ and $B_{d,0}$ are $O(\frac{\vartheta^2}{\rho_p})$ and so for the last term in the complexity bound \eqref{eq:prop:strict-unb} we have
\begin{eqnarray} \label{eq:trict-infea-comp-5-u}
O \left ( \sqrt{\vartheta}  \ln \left ( t_d(y^0)  B_{d,0} \right )\right)= O\left ( \sqrt{\vartheta}  \ln \left(\frac{\vartheta}{\rho_p}\right) \right).
\end{eqnarray}
Since $D=K-b$ and $\rec(D)=K$, for every $z\in D$ we have $z+b \in \rec(D)$.  We claim that the point $A \bar x$ in the statement of  Proposition \ref{prop:strict-unb} lies in $K$ (as well as $D=K-b$) in at most the same number of iterations as \eqref{eq:trict-infea-comp-5-u}, which gives an exact certificate of unboundedness. To prove this, we can use another shadow sequence $\hat x(\mu)$ as the optimal solution of \eqref{eq:dd-outcome-26} when we replace $\Phi(Ax)=\hat \Phi(Ax+b)$ with $\hat \Phi(Ax)$. Then, we can get a similar inequality as \eqref{eq:dd-outcome-27} and follow the same argument. This means that in the conic case, our detector returns an exact certificate of unboundedness in $O\left ( \sqrt{\vartheta}  \ln \left(\frac{\vartheta}{\rho_p}\right) \right)$ iterations, similar to \cite{infea-2}.
\end{remark}
\section{Ill-conditioned problems} \label{sec:ill}
The study of ``ill-posed" problems in \cite{infea-2} is restricted to a special case that $\hat \tau$ is negative, which provides a weak infeasibility detector. In this section, we see how the PtPCA algorithm performs when the problem instance is ill-conditioned (closed to be ill-posed). 


We say that problem \eqref{main-p} is $\epsilon$-feasible if there exists a point $(x,\tau,y) \in Q_{DD}$ such that $\frac{1}{\epsilon} \leq \tau$. 
Consider the case that both primal and dual problems are feasible, but the duality gap $\Lambda \neq0$. Let $ \bar x \in \mathcal F_p$ and $ \bar y \in \mathcal F_d$, with duality gap equal to $\langle c, \bar x \rangle + \delta_*(\bar y|D)=\bar \Lambda \geq \Lambda$. Using \eqref{eq:dd-outcome-3}, for every point $(x,\tau,y)$ $\kappa$-close to the central path, 
we have
\begin{eqnarray} \label{eq:dd-ill-1}
\frac{\mu}{\tau^2} \leq  \frac{1}{(\xi-1)\vartheta-\kappa \sqrt{\vartheta}} \left(\frac{\xi \vartheta + \langle y^0- \bar y, z^0-A \bar x  \rangle}{\tau_{\xi,\kappa}}+ \bar\Lambda \right).  
\end{eqnarray}
Therefore, we get approximately primal and dual feasible points for large values of $\mu$. 
\begin{lemma}\label{lem:feas-conv-1}
Assume that both primal and dual problems are feasible, i.e., $\mathcal F_p$ and $\mathcal F_d$ defined in \eqref{eq:feas-regions} are nonempty. Let $(x^k,\tau_k,y^k)$ be the sequence of points generated by PtPCA. Then,\\
(a) The PtPCA algorithm returns a pair of $\epsilon$-feasible primal-dual points in polynomial time. \\
(b) Both sequences $\{\langle c, x^k \rangle\}$ and $\{\frac{1}{\tau_k} \delta_* (y^k  |D)\}$ have accumulation points and for every pair of primal-dual feasible points $(x,y)$, accumulation points $ac_{p}$ of $\{\langle c, x^k \rangle\}$ and $ac_{d}$ of $\{\frac{1}{\tau_k} \delta_* (y^k  |D)\}$ satisfy
\begin{eqnarray} \label{eq:dd-ill-4-2}
-  \delta_* ( y |D) \leq ac_{p} \leq - ac_{d} \leq  \langle c,  x \rangle. 
\end{eqnarray}
\end{lemma}
\begin{proof}
Part (a) is implied by \eqref{eq:dd-ill-1}. For part (b), for every $y \in D_*$ and every $k$, by definition of $\delta_*$, we have 
\begin{eqnarray} \label{eq:dd-ill-4-12}
-\delta_* ( y |D) \leq - \langle  y ,  Ax^k +\frac{1}{\tau_k} z^0 \rangle = \langle c, x^k \rangle-\frac{1}{\tau_k} \langle y,z^0 \rangle,
\end{eqnarray}
where we used $A^\top y = -c$.  The fact that $\tau_k$ tends to $+\infty$ and \eqref{eq:dg-bound-1} imply that for $k$ large enough, we have
\begin{eqnarray} \label{eq:dd-ill-4-13}
\langle c, x^k \rangle  \leq -\frac{1}{\tau_k} \delta_* (y^k  |D). 
\end{eqnarray}
By using the definition of $\delta_*$ and $ \frac{1}{\tau_k} A^\top y^k =-c+\frac{1}{\tau_k} (A^\top y^0+c)$, for every $x$ with $Ax \in D$ we have
\begin{eqnarray} \label{eq:dd-ill-4-14}
- \frac{1}{\tau_k} \delta_* (y^k  |D) \leq - \langle \frac{y^k}{\tau_k},  A  x \rangle = \langle c,  x \rangle -\frac{1}{\tau_k} \langle A^\top y^0+c, x \rangle. 
\end{eqnarray}
By sending $k$ to $+\infty$ and using the fact that $\tau_k$ tends to $+\infty$, we get the desired results.   
\end{proof}
Note that every accumulation point of $\{\langle c, x^k \rangle\}$ gives a lower bound for the objective value of every feasible point. Therefore, we immediately have the following corollary:
\begin{corollary}
Assume that both primal and dual problems are feasible and the sequence $\{x^k\}$ generated by PtPCA has an accumulation point $\bar x$. Then, $\bar x$ is optimal for the primal problem. 
\end{corollary}
If both primal and dual problems are feasible, but the duality gap is not zero, we have the following proposition about the behavior of the output sequences of the algorithm. 
\begin{proposition} \label{prop:ill-1}
Assume that both primal and dual problems are feasible and  the duality gap between the primal and dual problems is $\Lambda \neq 0$. Let $(x^k,\tau_k,y^k)$ be the sequence of points generated by PtPCA. Then, we have
\[\limsup_{k\rightarrow+\infty} \max\{\|x^k\|,\|y^k/\tau_k\|\} = +\infty. \]  
\end{proposition}
\begin{proof}
For the sake of reaching a contradiction, assume that both $\{x^k\}$ and $\{y^k/\tau_k\}$ have accumulation points $\hat x$ and $\hat y$, which are primal and dual feasible, respectively, as $\tau_k$ tends to $+\infty$. Then, because $\tau_k$ tends to $+\infty$, \eqref{eq:dg-bound-1} implies that $\langle c,\hat x \rangle + \delta_*(\hat y|D) = 0$, which is a contradiction as the duality gap is positive. 
\end{proof}
Assume that we run the PtPCA algorithm until we get a point $(x,\tau,y) \in Q_{DD}$ $\kappa$-close to the central path with  $\mu \geq \frac{1}{\vartheta \epsilon^3}$, and the point is not detected as $\epsilon$-solution, or $\epsilon$-certificate of infeasibility or unboundedness. By the argument we had for the weak detector in Section \ref{sec:weak}, $\frac{\vartheta \mu}{\tau^2}$ is not large enough and we can see that $\tau=O(1/\epsilon)$. Therefore, problem \eqref{main-p}  is $O(\epsilon)$-feasible. We can argue that problem \eqref{main-p} is close to be ill-posed, but the exact categorization is impossible in the sense that all the ill-posed statuses we defined may lead to such an outcome. If there exists $\bar x \in \mathcal F_p$ and $\bar y \in \mathcal F_d$ with zero duality gap, then for every such pair by \eqref{eq:dd-outcome-3} we must have 
\begin{eqnarray} \label{eq:ill-cond-1}
\xi \vartheta + \langle y^0- \bar y, z^0-A \bar x  \rangle = O\left(\frac{1}{\epsilon}\right). 
\end{eqnarray}


 \section{Stopping criteria and conclusion} \label{sec:conclude}
Based on the insights we have gained by our performance analyses of the PtPCA algorithm in detecting the possible statuses for a given problem, we can discuss the stopping criteria and returned certificates by this algorithm in a practical setup.  Even though applications of interior-point methods beyond the scope of symmetric cones have been studied \cite{tunccel2001generalization,towards, skajaa-ye,myklebust2014interior, MR1953253}, there is no well-stablished software close to optimization in the Domain-Driven from. Let us review the existing stopping criteria for some well-known optimization over symmetric cones solvers (using the formulation in \eqref{intro-1}). For SDPT3 \cite{SDPT3-user-guide}, the algorithm is stopped for a given accuracy $\epsilon$ if at the current primal-dual point $(\hat z, (\hat v, \hat s))$:
\begin{enumerate}
\item an $\epsilon$-solution is obtained:
\begin{eqnarray*} \label{eq:stop-1}
\max \left\{\frac{|\langle \hat c, \hat z \rangle+ \langle \hat b, \hat v \rangle|}{1+|\langle \hat c, \hat z \rangle|+|\langle \hat b, \hat v \rangle|},  \frac{\|\hat A \hat z - \hat b\|}{1+\|\hat b\|},  \frac{\|\hat A ^\top \hat v+\hat c-\hat s\|}{1+\|\hat c\|}\right\} \leq \epsilon. 
\end{eqnarray*}
\item primal infeasibility is suggested:
\begin{eqnarray*} \label{eq:stop-2}
\frac{-\langle \hat b, \hat v \rangle}{\|\hat A ^\top \hat v-\hat s\|} > \frac{1}{\epsilon}.
\end{eqnarray*}
\item dual infeasibility is suggested:
\begin{eqnarray*} \label{eq:stop-3}
\frac{- \langle \hat c, \hat z \rangle}{\|\hat A  \hat z\|} > \frac{1}{\epsilon}.
\end{eqnarray*}
\item progress is slow, numerical problems are encountered, or the step sizes are small. 
\end{enumerate}
Freund in \cite{behavior} studied the (slightly modified) stopping criteria used by SeDuMi \cite{sedumi}, where an $\epsilon$-solution is suggested when 
\begin{eqnarray*} \label{eq:stop-8}
\frac{\max\{0, \langle \hat c, \hat z \rangle+ \langle \hat b, \hat v \rangle\}}{\max\{1,|\langle \hat c, \hat z \rangle|,|\langle \hat b, \hat v \rangle|\}}+  2\frac{\|\hat A \hat z - \hat b\|_\infty}{1+\|\hat b\|_\infty}+2  \frac{\|\hat A ^\top \hat v+\hat c-\hat s\|_\infty}{1+\|\hat c\|_\infty} \leq \epsilon. 
\end{eqnarray*}
We want to design our stopping criteria based on our analyses of different statuses of a problem in Domain-Driven formulation.  Compared to the existing stopping criteria in the literature, we  make more rigorous decisions on ill-posed problems. We define the following parameters for the current point $(x,\tau,y)$:
\begin{eqnarray} \label{eq:stop-4}
gap:=\frac{|\langle c,x\rangle+\frac{1}{\tau}\delta_*(y|D)|}{1+|\langle c,x\rangle|+|\frac{1}{\tau}\delta_*(y|D)|}, \ \ P_{feas}:= \frac{1}{\tau} \|z^0\|, \ \ D_{feas}:= \frac{\|\frac{1}{\tau} A^\top y + c\|}{1+\|c\|}. 
\end{eqnarray}
Let us consider the main statuses of a given problem:\\
\noindent {\bf Strictly primal-dual feasible or having unstable optima with dual certificate:} We return a point $x$ as an $\epsilon$-solution of the problem with an approximate certificate $y$ if $x$ and $y$ are approximately primal and dual feasible, while their duality gap is close to zero. In Section \ref{sec:solvable}, we proved that when $\mu_k:=\mu(x^k,\tau_k,y^k)$ tends to $+\infty$, $\tau_k$ also increases with a lower bound directly proportional to $\mu_k$ based on \eqref{eq:thm:fes-meas-dd-1} if the primal and dual are strictly feasible, or based on \eqref{eq:dd-outcome-3} if the problem has unstable optima with dual certificate. This implies that we have $P_{feas} \leq \epsilon$ and $D_{feas} \leq \epsilon$ in polynomial time. Also inequalities in \eqref{eq:dg-bound-1} guarantee $\langle c,x^k \rangle +\frac{1}{\tau_k}\delta_*(y^k|D) \leq \epsilon$ in polynomial time. The $gap$ defined in \eqref{eq:stop-4} is scaled, similar to many other practical software, to make the measure for duality gap scale independent. 
For the Domain-Driven algorithm, we say that an $\epsilon$-solution is obtained if 
\begin{eqnarray} \label{eq:stop-5}
\max\{ gap, P_{feas}, D_{feas}\} \leq \epsilon. 
\end{eqnarray}
\noindent {\bf Infeasible:} We studied infeasible statuses in Sections \ref{sec:weak} and \ref{sec:strict}. A point $y \in D_*$ is a certificate of infeasibility if $A^\top y=0$ and $\delta_*(y|D) <0$. We proved that if the problem is strongly or strictly infeasible, $\frac{\tau_k}{\mu_k} y^k$ becomes an $\epsilon$-certificate in polynomial time. We suggest that the problem is infeasible if for the current point $(x,\tau,y)$ we have
\begin{eqnarray} \label{eq:stop-6}
\frac{\tau}{\mu} \|A^\top y\| \leq \epsilon, \ \ \ \frac{\tau}{\mu} \delta_*(y|D) < 0. 
\end{eqnarray}
\noindent {\bf Unbounded:} We studied unbounded statuses in Sections \ref{sec:weak} and \ref{sec:strict}. Note that our approach here is different from many others for conic optimization. Instead of dual infeasibility, we return the unboundedness of the primal problem.  We suggest that the problem is unbounded if 
\begin{eqnarray} \label{eq:stop-7}
\langle c, x \rangle \leq - \frac{1}{\epsilon}, 
\end{eqnarray}
which can be done if the problem is strictly or strongly unbounded. \\
\noindent {\bf Ill-posed:} We studied these problems in Section \ref{sec:ill}. By our results, we know that if both primal and dual problems are feasible, the PtPAC algorithm eventually returns $\epsilon$-feasible solutions. Lemma \ref{lem:feas-conv-1} shows that $-\frac{1}{\tau_k} \delta_* (y^k  |D)$ gives the best estimate of optimal objective value. Our discussion in Section \ref{sec:ill} gives an idea how large $\mu$ should be in ill-posed cases so we can infer useful information about the problem. 

{\bf Stopping Criteria:} For a given tolerance $\epsilon$, run the algorithm until for the current point $(x,\tau,y)$ we have one of \eqref{eq:stop-5}, \eqref{eq:stop-6}, or \eqref{eq:stop-7}, or $\mu(x,\tau,y) \geq \frac{1}{\vartheta \epsilon^3}$. 

\begin{enumerate}
\item If \eqref{eq:stop-5} happens, return $(x,\frac{y}{\tau})$ as an $\epsilon$-solution. 
\item If \eqref{eq:stop-6} happens, return $\frac{\tau}{\mu} y$ as an $\epsilon$-certificate of infeasibility. 
\item If \eqref{eq:stop-7} happens, return $x$ as an $\epsilon$-certificate of unboundedness. 
\item If $\mu(x,\tau,y) \geq \frac{1}{\vartheta \epsilon^3}$ happens, then both primal and dual problems are approximately feasible ($\epsilon$-perturbations of the problems are feasible):
\begin{itemize}
\item $x$ and $\frac{y}{\tau}$ are $\epsilon$-feasible points for the primal and dual problems, respectively. 
\item $-\frac{1}{\tau}\delta_*(y|D)$ is the closest estimate to the optimal objective value. 
\end{itemize}
\end{enumerate}

\renewcommand{\baselinestretch}{1}
\bibliographystyle{siam}
\bibliography{References}

\end{document}